\documentclass[12pt]{amsart}
\pdfoutput=1

\usepackage[paperheight=11in, paperwidth=8.5in, left=1in, top=1.25in, right=1in, bottom=1.25in]{geometry}
\usepackage[linktocpage=true, linktoc=all, colorlinks=true, linkcolor=black, citecolor=black, filecolor=black, urlcolor=black, pagebackref=false, pdfstartpage={1}, pdfstartview={FitH}, pdftitle={}, pdfauthor={}, pdfsubject={}, pdfcreator={}, pdfproducer={}, pdfkeywords={}, bookmarksdepth=99]{hyperref}
\usepackage{afterpage}
\usepackage{amsfonts}
\usepackage{amsmath}
\allowdisplaybreaks[4]
\usepackage{amssymb}
\usepackage{amsthm}
\usepackage{mathtools}
\usepackage{array}
\usepackage{arydshln}
\usepackage{bbm}
\usepackage[justification=centering]{caption}
\usepackage[capitalize,nameinlink,noabbrev,compress]{cleveref}

\usepackage{enumerate}
\usepackage{enumitem}
\usepackage{float}
\restylefloat{figure}
\usepackage{graphicx}
\usepackage{kbordermatrix}
\usepackage{mathrsfs}
\usepackage[framemethod=tikz,ntheorem,xcolor]{mdframed}
\usepackage{multirow}
\usepackage{parcolumns}
\usepackage{scalerel}
\usepackage{tikz}\pdfpageattr{/Group <</S /Transparency /I true /CS /DeviceRGB>>}
\usetikzlibrary{arrows, calc, cd, decorations.markings, intersections, matrix, positioning, through}
\tikzset{commutative diagrams/.cd, every label/.append style = {font = \normalsize}}
\usepackage[all]{xy}
\usepackage[aligntableaux=center]{ytableau}

\numberwithin{equation}{section}
\newtheorem{theorem}{Theorem}
\AfterEndEnvironment{theorem}{\noindent\ignorespaces}
\numberwithin{theorem}{section}

\AfterEndEnvironment{conjecture}{\noindent\ignorespaces}
\newtheorem{corollary}[theorem]{Corollary}
\AfterEndEnvironment{corollary}{\noindent\ignorespaces}
\newtheorem{lemma}[theorem]{Lemma}
\AfterEndEnvironment{lemma}{\noindent\ignorespaces}
\newtheorem{problem}[theorem]{Problem}
\AfterEndEnvironment{problem}{\noindent\ignorespaces}
\newtheorem{proposition}[theorem]{Proposition}
\AfterEndEnvironment{proposition}{\noindent\ignorespaces}

\theoremstyle{definition}

\AfterEndEnvironment{definition}{\noindent\ignorespaces}
\newtheorem{example_no_qed}[theorem]{Example}
\newenvironment{example}[1][]{\begin{example_no_qed}[#1]\pushQED{\qed}}{\popQED\end{example_no_qed}}
\AfterEndEnvironment{example}{\noindent\ignorespaces}
\newtheorem{remark}[theorem]{Remark}
\AfterEndEnvironment{remark}{\noindent\ignorespaces}

\theoremstyle{remark}

\AfterEndEnvironment{claim}{\noindent\ignorespaces}
\newtheorem*{claim_proof_no_qed}{Proof of Claim}

\AfterEndEnvironment{claim_proof}{\noindent\ignorespaces}

\AddToHook{env/conjecture/begin}{\crefalias{theorem}{conjecture}}
\AddToHook{env/corollary/begin}{\crefalias{theorem}{corollary}}
\AddToHook{env/lemma/begin}{\crefalias{theorem}{lemma}}
\AddToHook{env/problem/begin}{\crefalias{theorem}{problem}}
\AddToHook{env/proposition/begin}{\crefalias{theorem}{proposition}}
\AddToHook{env/definition/begin}{\crefalias{theorem}{definition}}
\AddToHook{env/example/begin}{\crefalias{theorem}{example}}
\AddToHook{env/example_no_qed/begin}{\crefalias{theorem}{example_no_qed}}
\AddToHook{env/remark/begin}{\crefalias{theorem}{remark}}

\crefname{theorem}{Theorem}{Theorems}
\crefname{conjecture}{Conjecture}{Conjectures}
\crefname{corollary}{Corollary}{Corollaries}
\crefname{lemma}{Lemma}{Lemmas}
\crefname{problem}{Problem}{Problems}
\crefname{proposition}{Proposition}{Propositions}
\crefname{definition}{Definition}{Definitions}
\crefname{example}{Example}{Examples}
\crefname{example_no_qed}{Example}{Examples}
\crefname{remark}{Remark}{Remarks}

\renewcommand{\eqref}[1]{\hyperref[#1]{\textup{(\ref*{#1})}}}
\newcommand{\excess}{\operatorname{excess}}

\DeclareMathOperator{\Fl}{Fl}

\DeclareMathOperator{\GL}{GL}

\DeclareMathOperator{\Gr}{Gr}

\DeclareMathOperator{\Hom}{Hom}

\DeclareMathOperator{\SL}{SL}

\DeclareMathOperator{\spn}{span}

\DeclareMathOperator{\Stab}{Stab}

\newcommand{\sumof}[1]{\sum\hspace*{-1pt}{#1}}

\newcommand{\transpose}[1]{{#1}{\hspace*{-0.2pt}\raisebox{4pt}{$\scriptstyle\mathsf{T}$}\hspace*{0.5pt}}}
\newcommand{\transposed}[1]{{#1}^{\hspace*{-0.2pt}\scriptstyle\mathsf{T}}}

\DeclareMathOperator{\wt}{wt}

\newcommand{\CC}{\mathbb{C}}

\newcommand{\NN}{\mathbb{N}}

\newcommand{\PP}{\mathbb{P}}

\newcommand{\RR}{\mathbb{R}}

\newcommand{\cB}{\mathcal{B}}

\newcommand{\cO}{\mathcal{O}}
\newcommand{\cP}{\mathcal{P}}

\newcommand{\cT}{\mathcal{T}}

\newcommand{\framed}{\widehat{\cT}}
\newcommand{\newword}[1]{{\bfseries\itshape #1}}
\DeclareMathOperator{\oppborel}{\pi_{\textnormal{op}}}

\newcommand{\Rcell}[2]{\mathring{R}_{#1,#2}}
\newcommand{\Rtn}[2]{R_{#1,#2}^{<0}}
\newcommand{\Rtp}[2]{R_{#1,#2}^{>0}}
\newcommand{\Rvar}[2]{R_{#1,#2}}
\newcommand{\Scell}[1]{\mathring{X}_{#1}}
\newcommand{\Scellop}[1]{\mathring{X}_{#1}^{\textnormal{op}}}

\newcommand{\Svar}[1]{X_{#1}}
\newcommand{\Svarop}[1]{X_{#1}^{\textnormal{op}}}
\newcommand{\symgp}[1]{\mathfrak{S}_{#1}}

\title{Totally nonnegative maximal tori and opposed Bruhat intervals}

\author{Grant T. Barkley and Steven N. Karp}
\address{Department of Mathematics, University of Michigan}
\email{\href{mailto:gbarkley@umich.edu}{gbarkley@umich.edu}}
\address{Department of Mathematics, University of Notre Dame}
\email{\href{mailto:skarp2@nd.edu}{skarp2@nd.edu}}


\begin{document}

\begin{abstract}
Lusztig (2024) recently introduced the space $\mathcal{T}_{>0}$ of totally positive maximal tori of an algebraic group $G$. Each such torus is the intersection of a totally positive Borel subgroup and a totally negative Borel subgroup. Lusztig defined a map from the totally positive part of $G$ to $\mathcal{T}_{>0}$ and conjectured that it is surjective. We verify this conjecture. We also examine the closure of $\mathcal{T}_{>0}$, by studying when a totally nonnegative Borel subgroup is opposed to a totally nonpositive Borel subgroup. Our main result reduces this problem to a new combinatorial relation between pairs of Bruhat intervals of the Weyl group $W$, which we call `opposition'. We provide a characterization of opposition when $G = \text{SL}_n$ (and $W$ is the symmetric group). Along the way, we disprove another conjecture of Lusztig (2021) on totally nonnegative Borel subgroups. Finally, we connect $\mathcal{T}_{>0}$ to the amplituhedron introduced by Arkani-Hamed and Trnka (2014) in theoretical physics, by showing that $\mathcal{T}_{>0}$ can be regarded as a `universal flag amplituhedron'. This gives further motivation for studying $\mathcal{T}_{>0}$ and its closure.
\end{abstract}

\maketitle
\setcounter{tocdepth}{1}
\tableofcontents

\section{Introduction}\label{sec:introduction}

\noindent The theory of total positivity has been studied for over a century. A classical object in this area is a \emph{totally positive matrix}, namely, a real $n\times n$ matrix whose minors are all positive. Such matrices turn out to have remarkable properties: they can be parametrized using networks \cite{fomin_zelevinsky00}, and their complex eigenvalues are real, positive, and distinct \cite{gantmakher_krein37}. Recent developments have focused on total positivity for groups, flag varieties, cluster algebras, and related objects, with applications to several areas of mathematics \cite{lusztig94,fomin_zelevinsky02,postnikov06}.

\subsection{Totally positive maximal tori}
Let $\cT$ denote the space of maximal tori of an algebraic group $G$. Lusztig \cite{lusztig} recently introduced the totally positive part $\cT_{>0}$ of $\cT$, as follows. Let $\cB$ be the complete flag variety of all Borel subgroups of $G$, and let $\cB_{>0}$ and $\cB_{<0}$ denote its totally positive and totally negative parts, respectively. We call two Borel subgroups \emph{opposed} if their intersection is a maximal torus of $G$. It turns out that every Borel subgroup in $\cB_{>0}$ is opposed to every Borel subgroup in $\cB_{<0}$, and $\cT_{>0}$ is defined to be the set of such intersections:
\[
\cT_{>0} \coloneqq \{B \cap B' \mid B\in\cB_{>0} \text{ and } B'\in\cB_{<0}\}.
\]

To be concrete, we explain what this means when $G = \SL_n(\CC)$. In this case, we can identify $\cB$ with $\Fl_n(\CC)$, the space of tuples of complete flags
\[
F_\bullet = (0 \subset F_1 \subset \cdots \subset F_{n-1} \subset \CC^n),
\]
where each $F_k$ is a $k$-dimensional subspace of $\CC^n$. Then $\cB_{>0}$ consists of all complete flags $F_\bullet$ such that every $F_k$ has positive Pl\"{u}cker coordinates (up to rescaling), and $\cB_{<0}$ consists all complete flags $F'_\bullet$ such that every $F'_k$ has positive Pl\"{u}cker coordinates (up to rescaling) after we replace the standard basis $(e_1, \dots, e_n)$ of $\CC^n$ with the re-signed basis $(e_1, -e_2, e_3, \dots, (-1)^{n-1}e_n)$. Also, every maximal torus $T$ of $G$ is uniquely determined by a basis $v_1, \dots, v_n$ for $\CC^n$ of common eigenvectors for the matrices in $T$. Then $T\in\cT_{>0}$ if and only if the basis vectors can be ordered such that the complete flag $F_\bullet$ generated by $v_1, \dots, v_n$ is totally positive and the complete flag $F'_\bullet$ generated by $v_n, \dots, v_1$ is totally negative.
\begin{example}\label{eg:lusztig_intro}
Let $G = \SL_3(\CC)$, and set
\[
v_1 = \begin{bmatrix}1 \\ 1 \\ 1\end{bmatrix},\hspace*{4pt}
v_2 = \begin{bmatrix}-0.5 \\ 0 \\ 1\end{bmatrix},\hspace*{4pt}
v_3 = \begin{bmatrix}0.4 \\ -0.2 \\ 0.6\end{bmatrix},\hspace*{4pt}
g = \begin{bmatrix}
\vline & \vline & \vline \\
v_1 & v_2 & v_3 \\
\vline & \vline & \vline \\
\end{bmatrix}
=
\begin{bmatrix}
1 & -0.5 & 0.4 \\
1 & 0 & -0.2 \\
1 & 1 & 0.6
\end{bmatrix}.
\]
The corresponding maximal torus $T$ is the set of matrices with eigenvectors $v_1$, $v_2$, and $v_3$:
\[
T = \left\{g\begin{bmatrix}\lambda_1 & 0 & 0 \\ 0 & \lambda_2 & 0 \\ 0 & 0 & \lambda_3\end{bmatrix}g^{-1} \;\middle\vert\; \lambda_1\lambda_2\lambda_3 = 1\right\} \subseteq G.
\]
We claim that $T$ is totally positive, i.e., $T\in\cT_{>0}$.

Indeed, the flag $F_\bullet\in\Fl_3(\CC)$ generated by $v_1, v_2, v_3$ is totally positive. To see this, note that the Pl\"{u}cker coordinates of $F_1 = \spn(v_1)$ are the entries of $v_1$ (i.e.\ $1$, $1$, and $1$), which are all positive. Also, the Pl\"{u}cker coordinates of $F_2 = \spn(v_1,v_2)$ are the $2\times 2$ minors of $g$ using the first two columns (i.e.\ $0.5$, $1.5$, and $1$), which are also all positive.

Similarly, the flag $F'_\bullet\in\Fl_3(\CC)$ generated by $v_3, v_2, v_1$ is totally negative. To see this, note that $F'_1 = \spn(v_3)$, and its re-signed Pl\"{u}cker coordinates are the entries of $v_3$ with the second entry negated (i.e.\ $0.4$, $0.2$, and $0.6$), which are all positive. Also, $F'_2 = \spn(v_3, v_2)$, and its re-signed Pl\"{u}cker coordinates are the $2\times 2$ minors of $g$ (with its second row negated) using the last two columns (i.e.\ $-0.1$, $-0.7$, and $-0.2$) which are all positive (up to simultaneous rescaling). Therefore $T\in\cT_{>0}$.
\end{example}

Following \cite{lusztig}, one way to construct elements of $\cT_{>0}$ is to use the totally positive part $G_{>0}$ of $G$. Namely, every $h\in G_{>0}$ is contained in a unique $B\in\cB_{>0}$ and a unique $B'\in\cB_{<0}$. This gives rise to a map $\pi' : G_{>0} \to \cT_{>0}$ sending $h \mapsto B\cap B'$. When $G = \SL_n(\CC)$, we can interpret $\pi'$ explicitly as follows. Recall that every $h\in\SL_n^{>0}$ has $n$ distinct real eigenvalues $\lambda_1 > \cdots > \lambda_n > 0$; let $v_1, \dots, v_n\in\CC^n$ be the corresponding eigenvectors. Then $\pi'(h)$ is the maximal torus consisting of all matrices with eigenvectors $v_1, \dots, v_n$. For example, for the torus $T\in\cT_{>0}$ from \cref{eg:lusztig_intro}, we have $T = \pi'(h)$, where
\[
h = \begin{bmatrix}1.48 & 3.28 & 0.24 \\ 0.96 & 3.56 & 0.48 \\ 0.32 & 3.52 & 1.16\end{bmatrix} \in \SL_3^{>0}.
\]
Indeed, $h$ has eigenvalues $5$, $1$, and $\frac{1}{5}$ with corresponding eigenvectors $v_1$, $v_2$, and $v_3$.

Our first main result says that every element of $\cT_{>0}$ can be constructed using $G_{>0}$:
\begin{theorem}\label{intro_conjecture_consequence}
The map $\pi' : G_{>0} \to \cT_{>0}$ is surjective.
\end{theorem}

\cref{intro_conjecture_consequence} verifies a conjecture of Lusztig from \cite[Section 5]{lusztig}. In fact, the precise statement of Lusztig's conjecture is slightly stronger and more technical, and is provided in \cref{lusztig_conjecture}. We prove this stronger form of Lusztig's conjecture in \cref{conjecture_proof}.

As a by-product of our methods, we also obtain the following result about the totally nonnegative parts of $G$ and $\cB$ (see \cref{lusztig_counterexample}):
\begin{proposition}\label{intro_lusztig_counterexample}
When $G = \SL_3(\CC)$, there exists $B\in\cB_{\ge 0}$ which does not contain any regular semisimple element of $G_{\ge 0}$.
\end{proposition}

\cref{intro_lusztig_counterexample} provides a counterexample to a different conjecture of Lusztig (from \cite[Section 5.6]{lusztig21}).

\subsection{Opposition for Bruhat intervals}
We now introduce the space $\cT_{\ge 0}$ of totally nonnegative maximal tori, which we define to be the closure of $\cT_{>0}$ in the Euclidean topology. It turns out that
\[
\cT_{\ge 0} = \{B\cap B' \mid B\in\cB_{\ge 0}, B'\in\cB_{\le 0}, \text{ and } B \text{ is opposed to } B'\}.
\]
Recall that every Borel subgroup in $\cB_{>0}$ is opposed to every Borel subgroup in $\cB_{<0}$. However, this is no longer true when we replace $\cB_{>0}$ with $\cB_{\ge 0}$ and $\cB_{<0}$ with $\cB_{\le 0}$. (For example, the standard Borel subgroup $B_+$ lies in both $\cB_{\ge 0}$ and $\cB_{\le 0}$, but the intersection $B_+ \cap B_+ = B_+$ is far from being a maximal torus.) Therefore the fundamental problem in studying $\cT_{\ge 0}$ is the following:
\begin{problem}\label{intro_opposed_problem}
When are two Borel subgroups $B\in\cB_{\ge 0}$ and $B'\in\cB_{\le 0}$ opposed to each other?
\end{problem}

Our work shows that \cref{intro_opposed_problem} is surprisingly subtle and deep. Our first result addressing \cref{intro_opposed_problem} uses the cell decomposition of the totally nonnegative part $\cB_{\ge 0}$ of $\cB$. Recall from \cite{lusztig94,rietsch99} that
\[
\cB_{\ge 0} = \bigsqcup_{v\le w}\Rtp{v}{w},
\]
where the disjoint union is over all $v\le w$ in the Bruhat order on the Weyl group $W$ of $G$, and $\Rtp{v}{w}$ denotes the totally positive part of the open Richardson variety $\Rcell{v}{w}$ of $\cB$, which is the intersection of the opposite Schubert cell indexed by $v$ and the Schubert cell indexed by $w$ \cite{kazhdan_lusztig79}. Each $\Rtp{v}{w}$ is an open cell of dimension equal to the length of the Bruhat interval $[v,w]$. We have a similar decomposition of $\cB_{\le 0}$ indexed by Bruhat intervals of $W$, where $[v,w]$ labels the cell of $\cB_{\le 0}$ obtained by intersecting with the open Richardson variety $\Rcell{ww_0}{vw_0}$, where $w_0$ denotes the longest element of $W$.

Our second main result shows that opposition between Borel subgroups in $\cB_{\ge 0}$ and $\cB_{\le 0}$ only depends on their underlying cells (see \cref{opposition_cells}):
\begin{theorem}\label{intro_opposition_cells}
Let $B\in \cB_{\ge 0}$ and $B'\in \cB_{\le 0}$. Then whether $B$ and $B'$ are opposed depends only on the pair of cells containing $B$ and $B'$.
\end{theorem}

Motivated by \cref{intro_opposition_cells}, we say that two Bruhat intervals $[v,w]$ and $[v',w']$ of $W$ are \newword{opposed} if some/every element of the cell of $\cB_{\ge 0}$ labeled by $[v,w]$ is opposed to some/every element of the cell of $\cB_{\le 0}$ labeled by $[v',w']$. This defines an interesting new combinatorial relationship between Bruhat intervals of $W$.

We mention that Richardson varieties have been extensively studied due to their connections with Kazhdan--Lusztig theory \cite{kazhdan_lusztig79}, Schubert calculus \cite{speyer}, total positivity \cite{marsh_rietsch04}, and cluster algebras \cite{ingermanson19,casals_gorsky_gorsky_le_shen_simental25,galashin_lam_sherman-bennett26}. This provides additional motivation for studying opposition between two Richardson cells.

Our third main result characterizes opposition for Bruhat intervals in type $A$ (see \cref{opposition_A}):
\begin{theorem}\label{intro_opposition_A}
Let $G = \SL_n(\CC)$, so that $W$ is the symmetric group $\symgp{n}$ of all permutations of $n$. Then the Bruhat intervals $[v,w]$ and $[v',w']$ of $W$ are opposed if and only if for all $1 \le k \le n-1$, there exist $x\in [v,w]$ and $x'\in [v,w]$ such that $x(\{1, \dots, k\}) = x'(\{1, \dots, k\})$.
\end{theorem}

For example, we can use \cref{intro_opposition_A} to verify that the Bruhat intervals $[132,231]$ and $[213,312]$ of $\symgp{3}$ are opposed; see \cref{eg:opposition_S3} for the details.

While we are unable to characterize opposition for general Weyl groups $W$, we establish a sufficient condition and a necessary condition (see \cref{intervals_intersect}, \cref{tnn_tp_opposed}, and \cref{BIPs}):
\begin{theorem}\label{intro_intervals_intersect}~
\begin{enumerate}[label=(\roman*), leftmargin=*, itemsep=2pt]
\item\label{intro_intervals_intersect_sufficient}
If two Bruhat intervals of $W$ intersect, then they are opposed. In particular, every Borel subgroup in $\cB_{>0}$ is opposed to every Borel subgroup in $\cB_{\le 0}$.
\item\label{intro_intervals_intersect_necessary} If two Bruhat intervals of $W$ are opposed, then their Bruhat interval polytopes intersect.
\end{enumerate}
\end{theorem}

We also generalize the notion of opposition between pairs of Borel subgroups of $G$ to pairs of parabolic subgroups. Our final main result (see \cref{opposition_maximal_parabolics}) reduces opposition for Borel subgroups to opposition for maximal parabolic subgroups:
\begin{theorem}\label{intro_opposition_maximal_parabolics}
    Let $I$ denote the set of simple roots of $G$, and let $i \leftrightarrow i^*$ denote the standard involution on $I$. Then two Borel subgroups $B$ and $B'$ are opposed if and only if for all $i\in I$, the maximal parabolic subgroup of type $I\setminus\{i\}$ containing $B$ is opposed to the maximal parabolic subgroup of type $I\setminus\{i^*\}$ containing $B'$.
\end{theorem}

One novelty of our arguments is that rather than use Lusztig's canonical basis (as in, e.g., \cite{lusztig94}), we use the Mirkovi\'{c}--Vilonen basis from \cite{baumann_kamnitzer_knutson21}, which is well-defined and has the desired positivity properties for all $G$ (even in non-simply-laced types, unlike the canonical basis). We call such a basis a \emph{positive weight basis} (see \cref{sec:positive_bases}). This allows us to give uniform proofs for both simply- and non-simply-laced types, without resorting to the usual folding technique to reduce to the simply-laced case.

We also mention that \cref{intro_opposed_problem} can be reduced to determining which basis coordinates in a positive weight basis are nonvanishing on a given Richardson cell $\Rtp{v}{w}$; see \cref{problem_coordinates}.

\subsection{Connection with amplituhedra}
Arkani-Hamed and Trnka \cite{arkani-hamed_trnka14} introduced the \emph{amplituhedron} in order to encode scattering processes of particles in high-energy physics. Following \cite{karp_williams19}, we can define it as follows. Let $\Gr_{k,n}(\CC)$ denote the \emph{Grassmannian} of all $k$-dimensional subspaces $V \subseteq \CC^n$. Its totally positive part $\Gr_{k,n}^{>0}$ consists of all $V$ whose Pl\"{u}cker coordinates are all positive (up to rescaling), and its totally nonpositive part $\Gr_{k,n}^{\le 0}$ consists of all $V$ whose Pl\"{u}cker coordinates are all nonnegative (up to rescaling) after we replace the standard basis $(e_1, \dots, e_n)$ of $\CC^n$ with the re-signed basis $(e_1, -e_2, e_3, \dots, (-1)^{n-1}e_n)$. Then given $W\in\Gr_{k+m,n}^{>0}$, the amplituhedron $\mathcal{A}_{n,k,m}(W)$ is defined to be
\begin{align}\label{intro_amplituhedron_definition}
\{W\cap V \mid V\in\Gr_{n-k,n}^{\le 0}\} \subseteq \Gr_{m,n}(\CC).
\end{align}
The amplituhedron is (conjecturally) a \emph{positive geometry} \cite{arkani-hamed_bai_lam17} whose canonical differential form is a tree-level scattering amplitude in planar $\mathcal{N}=4$ SYM theory (when $m=4$).

The fact that the amplituhedron in \eqref{intro_amplituhedron_definition} is well-defined reduces to showing that the intersection of an element in $\Gr_{k+m,n}^{>0}$ and an element in $\Gr_{n-k,n}^{\le 0}$ has the smallest possible dimension $m$. We can phrase this fact as saying that every element in $\Gr_{k+m,n}^{>0}$ is opposed (i.e.\ transverse) to every element in $\Gr_{n-k,n}^{\le 0}$. This is a Grassmannian analogue of the fact that every element in $\cB_{>0}$ is opposed to every element in $\cB_{\le 0}$ (as proved in \cref{intro_intervals_intersect}\ref{intro_intervals_intersect_sufficient}).

We are thus led to consider the flag analogue of \eqref{intro_amplituhedron_definition}. Namely, given $B\in\cB_{>0}$, we define the \emph{flag amplituhedron of $B$} to be
\begin{align}\label{intro_flag_amplituhedron}
\{B\cap B' \mid B'\in\cB_{\le 0}\} \subseteq \cT_{\ge 0}.
\end{align}
This line of reasoning makes it natural to study the space of totally nonnegative maximal tori (which we can regard as a sort of `universal' flag amplituhedron). We will use our results to show that the flag amplituhedron \eqref{intro_flag_amplituhedron} is in fact homeomorphic to $\cB_{\le 0}$ (see \cref{flagtope_bijection}).

Moreover, Lam \cite{lam16} proposed generalizing \eqref{intro_amplituhedron_definition} in a different way. Namely, $\Gr_{n-k,n}^{\le 0}$ has a decomposition into positroid cells; let $C$ denote the closure of such a cell. Given $W\in\Gr_{k+m,n}(\CC)$, Lam defines the \emph{Grassmann polytope}
\begin{align}\label{intro_grassmann_polytope}
\{W\cap V \mid V\in C\} \subseteq \Gr_{m,n}(\CC).
\end{align}
The Grassmann polytope \eqref{intro_grassmann_polytope} is only well-defined if every intersection $W\cap V$ has the smallest possible dimension $m$. This raises the following fundamental problem in the study of Grassmann polytopes:
\begin{problem}\label{intro_grassmann_polytope_problem}
Let $W\in\Gr_{k+m,n}(\CC)$, and let $C$ be a closed cell of $\Gr_{n-k,n}^{\le 0}$. When is the Grassmann polytope \eqref{intro_grassmann_polytope} well-defined? That is, when is $W$ opposed (i.e.\ transverse) to every element in $C$?
\end{problem}

Note that when $W$ is totally nonnegative, \cref{intro_grassmann_polytope_problem} is nothing but the Grassmannian analogue of \cref{intro_opposed_problem}, i.e., determining when two Bruhat intervals are opposed. We hope that the study of opposed Bruhat intervals will shed new light on Grassmann polytopes.

\subsection{Outline}
In \cref{sec:background} we recall some background on representation theory and total positivity. In \cref{sec:grassmannians} we study the notion of opposition (i.e.\ transversality) for subspaces, which will serve as a warmup for our more general results to follow. In \cref{sec:opposition_parabolic} we introduce opposition for pairs of parabolic subgroups, and show that determining opposition for Borel subgroups reduces to studying opposition for maximal parabolic subgroups. In \cref{sec:opposition_combinatorial} we prove that opposition for Borel subgroups depends only on the underlying Richardson cells, which leads us to define the notion of opposition between Bruhat intervals of $W$. In \cref{sec:opposition_combinatorics} we prove several results about opposition on Bruhat intervals, including a complete characterization when $W$ is the symmetric group $\mathfrak{S}_n$. In \cref{sec:lusztig_proof} we prove Lusztig's conjecture that the map $\pi' : G_{>0} \to \cT_{>0}$ is surjective. In \cref{sec:topology} we study the topology of the space $\cT_{\ge 0}$ of totally nonnegative maximal tori. In \cref{sec:counterexamples} we provide a counterexample to a different conjecture of Lusztig about $\cB_{\ge 0}$. Finally, in \cref{sec:amplituhedra} we relate total positivity for maximal tori to amplituhedra.

\subsection*{Acknowledgments}
We thank George Lusztig for comments on a preliminary version of the paper. G.T.B.\ was supported by the National Science Foundation under Award Nos.\ 2152991 and 2503536. S.N.K.\ was partially supported by the National Science Foundation under Award No.\ 2452061, by a travel support gift from the Simons Foundation, and by a grant from the Institute for Advanced Study School of Mathematics.

\section{Background}\label{sec:background}

\noindent We recall some background on total positivity for algebraic groups, following \cite{humphreys75,lusztig94,bjorner_brenti05,galashin_karp_lam22,lusztig}. Let $\NN \coloneqq \{0, 1, 2, \dots\}$ and define $[n] \coloneqq \{1, \dots, n\}$ for $n\in\NN$.

\subsection{Algebraic groups}
Let $G$ be a semisimple and simply connected algebraic group which is split over $\RR$; we will freely identify $G$ with its complex points, which form a complex Lie group. Fix an $\RR$-split maximal torus $T_0$ of $G$, and let $B_+, B_-\subseteq G$ be Borel subgroups such that $T_0 = B_+ \cap B_-$. We let $\cdot$ denote the action of $G$ on itself by conjugation, so for example $g\cdot S = \{ghg^{-1}\mid h\in H\}$ for every $S\subseteq G$.

Let $X(T_0) \coloneqq \Hom(T_0,\CC^\times)$ denote the \newword{weight lattice}, and let $\Phi \subseteq X(T_0)$ denote the set of \newword{roots}. Our choice of $B_+$ and $B_-$ induces the decomposition $\Phi = \Phi_+ \sqcup \Phi_-$, where $\Phi_+$ is the set of \newword{positive roots} and $\Phi_-$ is the set of \newword{negative roots}. Let $\{\alpha_i\mid i\in I\} \subseteq \Phi_+$ denote the set of \newword{simple roots}, where $I$ is an indexing set. Let $U_+$ and $U_-$ be the unipotent radicals of $B_+$ and $B_-$, respectively. Each  positive root $\alpha$ gives rise to one-parameter \newword{root subgroups} $U_\alpha \subseteq U_+$ and $U_{-\alpha}\subseteq U_-$.

For every $i\in I$ we pick a homomorphism $\phi_i: \SL_2(\CC) \to G$, and define
\[
x_i(t) \coloneqq \phi_i\!\left(\begin{bmatrix}1 & t \\ 0 & 1\end{bmatrix}\right),
\quad
y_i(t) \coloneqq \phi_i\!\left(\begin{bmatrix}1 & 0 \\ t & 1\end{bmatrix}\right),
\quad
\dot{s}_i \coloneqq \phi_i\!\left(\begin{bmatrix}0 & -1 \\ 1 & 0\end{bmatrix}\right).
\]
We require our choice of $\phi_i$ to satisfy $x_i(\CC) = U_{\alpha_i}$ and $y_i(\CC) = U_{-\alpha_i}$. We call the data $(T_0, B_+, B_-, I, \{x_i\mid i\in I\}, \{y_i\mid i\in I\})$ a \newword{pinning} for $G$. We also let $\transpose{\cdot}$ denote the involutive anti-automorphism of $G$ which satisfies
\begin{align}\label{lie_transpose}
\transpose{g} = g \text{ for all } g\in T_0
\quad \text{ and } \quad
\transpose{x_i(t)} = y_i(t) \text{ for all $i\in I$ and $t\in\CC$}.
\end{align}

We let $W \coloneqq N(T_0)/T_0$ denote the \newword{Weyl group} of $G$. Let $s_i \in W$ denote the element of $W$ represented by $\dot s_i$, so that $W$ is a Coxeter group with simple generators $\{s_i\mid i\in I\}$ (see \cite{bjorner_brenti05} for details). A \newword{reduced expression} for $w\in W$ is a minimal-length expression $w = s_{i_1} \cdots s_{i_l}$ as a product of simple generators; we call $l$ the \newword{length} of $w$, denoted $\ell(w)$. We define the group representative
\[
\dot{w} \coloneqq \dot{s}_{i_1} \cdots \dot{s}_{i_l} \in G,
\]
which does not depend on the choice of reduced expression for $w$. We have
\begin{align}\label{inverse_permutation_matrix}
(\dot{w})^{-1} = \transpose{\dot{w}} \quad \text{ for all } w\in W.
\end{align}

We let $\le$ denote the \newword{Bruhat order} on $W$, i.e., $v\le w$ if and only if $v$ has a reduced expression which is a subword of some reduced expression for $w$. The Bruhat order has a minimum $e$ (where $\dot{e}$ is the identity element of $G$) and the maximum $w_0$. For $v,w\in W$, we define the \newword{Bruhat interval} $[v,w] \coloneqq \{x\in W\mid v \le x \le w\}$. We also define the involution $\cdot^*$ on $I$ such that
\begin{align}\label{star_involution}
s_{i^*} = w_0s_iw_0 \quad\text{ for all } i\in I.
\end{align}
For $J\subseteq I$, we define $J^* \coloneqq \{i^* \mid i\in J\}$.

We define the \newword{coweight lattice} $Y(T) \coloneqq \Hom(\CC^\times, T)$. For all $i\in I$, we have the coroot $\alpha^\vee_i$ defined by
\[
\alpha^\vee_i(t) \coloneqq \phi_i\!\left(\begin{bmatrix}t & 0 \\ 0 & t^{-1}\end{bmatrix}\right) \quad \text{ for all } t\in\CC^\times.
\]

\begin{example}\label{eg:pinning}
We recall the standard pinning for $G = \SL_n(\CC)$, the group of invertible $n\times n$ matrices with determinant $1$. We let $B_+$ be the subgroup of upper-triangular matrices and $B_-$ be the subgroup of lower-triangular matrices, so that $T_0 = B_+ \cap B_-$ is the subgroup of diagonal matrices. We have $I = [n-1]$, and for $i\in I$ the map $\phi_i$ embeds $\SL_2(\CC)$ in rows and columns $\{i,i+1\}$. For example,
\[
x_i(t) = \kbordermatrix{& & i & i+1 & \cr
& \ddots & & & \cr
i & & 1 & t & \cr
i+1 & & 0 & 1 & \cr
& & & & \ddots}
\quad\text{ and }\quad
\alpha^\vee_i(t) = \kbordermatrix{& & i & i+1 & \cr
& \ddots & & & \cr
i & & t & 0 & \cr
i+1 & & 0 & t^{-1} & \cr
& & & & \ddots},
\]
where the matrices agree with the identity matrix $I_n$ in all unspecified entries. The map $\transpose{\cdot}$ is the usual matrix transpose.

The Weyl group $W$ is the symmetric group $\symgp{n}$ of all permutations of $[n]$. For $w\in\symgp{n}$, the element $\dot{w}\in\SL_n(\CC)$ is the signed permutation matrix satisfying
\[
\dot{w}_{i,j} = \pm\delta_{i,w(j)} \quad \text{ for all } i,j\in [n],
\]
where the signs $\pm$ are chosen uniquely so that all left-justified minors of $\dot{w}$ are nonnegative. For example, if $w = 312\in\symgp{3}$ then
\[
\dot{w} = \begin{bmatrix}
0 & -1 & 0 \\
0 & 0 & -1 \\
1 & 0 & 0
\end{bmatrix} \in\SL_3(\CC).
\]
The permutation $w_0$ is the involution sending $i$ to $n+1-i$ for all $i\in [n]$. For all $i\in [n-1]$, the simple generator $s_i$ is the transposition swapping $i$ and $i+1$, and $i^* = n-i$.
\end{example}

\subsection{Complete flag varieties}
Let $\cB$ denote the \newword{(complete) flag variety} of all Borel subgroups of $G$. Note that we can write every element of $\cB$ as
\[
g\cdot B_+ = gB_+g^{-1} \quad \text{ for some } g\in G.
\]
This allows us to identify $\cB$ with the quotient $G/B_+$ via $g\cdot B_+ \leftrightarrow gB_+$. Note that $B_- = \dot{w}_0\cdot B_+$. We define the involution $\cdot^\perp$ on $\cB$ by
\begin{align}\label{perp_group}
(g\cdot B_+)^\perp \coloneqq \transpose{(g^{-1})}\dot{w}_0\cdot B_+ \quad \text{ for all } g\in G,
\end{align}
which is well-defined because $\transpose{((B_+)^{-1})}\dot{w_0} = B_-\dot{w}_0 = \dot{w}_0B_+$.

\begin{example}\label{eg:flag_variety}
We continue the setup of \cref{eg:pinning}, where $G = \SL_n(\CC)$. Let $\Fl_n$ denote the set of \newword{complete flags} in $\CC^n$, i.e., tuples
\[
F_\bullet = (0 \subset F_1 \subset \cdots \subset F_{n-1} \subset \CC^n)
\]
where each $F_k$ is a $k$-dimensional subspace of $\CC^n$. Then we can identify $\cB$ with $\Fl_n$ by sending the Borel subgroup $g\cdot B_+$ (where $g\in\SL_n(\CC)$) to the complete flag $F_\bullet$, where each $F_k$ is spanned by the first $k$ columns of $g$. The inverse map takes a flag $F_\bullet$ to its stabilizer subgroup, which is a Borel subgroup of $G$. For example, if
\begin{align}\label{plucker_vector1}
g = \begin{bmatrix}
1 & 0 & 0 \\
2 & 1 & 0 \\
3 & 4 & 1
\end{bmatrix}\!,
\text{ then }
F_1 = \spn\left(\begin{bmatrix}1 \\ 2 \\ 3\end{bmatrix}\right)
\text{ and }
F_2 = \spn\left(\begin{bmatrix}1 \\ 2 \\ 3\end{bmatrix},
\begin{bmatrix}0 \\ 1 \\ 4\end{bmatrix}\right).
\end{align}

We also point out that the map \eqref{perp_group} is just given by taking orthogonal complements (see, e.g., \cite[Lemma 7.1]{karp_precup}). That is, given a $k$-dimensional subspace $V$ of $\CC^n$, let $V^\perp$ denote the $(n-k)$-dimensional subspace orthogonal to $V$ under the bilinear pairing for which the standard basis of $\CC^n$ is an orthonormal basis. Then
\begin{align}\label{perp_complete_flag}
(F_\bullet)^\perp = (0 \subset F_{n-1}^\perp \subset \cdots \subset F_1^\perp \subset \CC^n)
\end{align}
for all complete flags $F_\bullet$. Continuing the example above, we have
\begin{align}\label{plucker_vector2}
\transpose{(g^{-1})}\dot{w}_0 =
\transposed{\begin{bmatrix}
1 & 0 & 0 \\
-2 & 1 & 0 \\
5 & -4 & 1
\end{bmatrix}}
\begin{bmatrix}
0 & 0 & 1 \\
0 & -1 & 0 \\
1 & 0 & 0
\end{bmatrix}
=
\begin{bmatrix}
5 & 2 & 1 \\
-4 & -1 & 0 \\
1 & 0 & 0
\end{bmatrix}.
\end{align}
We can check that $F_{3-k}^\perp$ (for $k=1,2$) is spanned by the first $k$ columns of the matrix above. 
\end{example}

We now recall the Richardson varieties introduced by Kazhdan and Lusztig \cite{kazhdan_lusztig79}; see the survey \cite{speyer} for more details when $G = \SL_n(\CC)$. The \newword{Schubert cells} and \newword{opposite Schubert cells} are the $B_+$-orbits and $B_-$-orbits of $\cB$, respectively, and are indexed by $W$:
\[
\Scell{w} \coloneqq B_+\dot{w}\cdot B_+ \subseteq \cB
\quad \text{ and } \quad
\Scellop{w} \coloneqq B_-\dot{w}\cdot B_+ \subseteq \cB
\]
for $w \in W$. These affine cells of dimension $\ell(w)$ and codimension $\ell(w)$, respectively. We denote their Zariski closures in $\cB$ as
\[
\Svar{w} \coloneqq \overline{\Scell{w}} \quad \text{ and } \quad \Svarop{w} \coloneqq \overline{\Scellop{w}},
\]
called a \newword{Schubert variety} and an \newword{opposite Schubert variety}, respectively. For example, $\Svar{w_0} = \Svarop{e} = \cB$ and $\Svar{e} = \{\dot{e}\cdot B_+\}$.

For $v,w\in W$, we define the \newword{open Richardson variety} and the \newword{Richardson variety} in $\cB$ by
\begin{align}\label{richardson_defn}
\Rcell{v}{w} \coloneqq \Scellop{v} \cap \Scell{w} \quad \text{ and } \quad \Rvar{v}{w} \coloneqq \overline{\Rcell{v}{w}} = \Svarop{v} \cap \Svar{w},
\end{align}
respectively. We have
\begin{align}\label{richardson_nonempty}
\Rcell{v}{w} \neq \varnothing
\quad\Leftrightarrow\quad
\Rvar{v}{w}\neq \varnothing
\quad\Leftrightarrow\quad
v \le w \text{ in Bruhat order},
\end{align}
in which case $\Rvar{v}{w}$ is an irreducible projective variety of dimension $\ell(w) - \ell(v)$. In particular, Richardson varieties are indexed by the Bruhat intervals of $W$.

\begin{lemma}\label{perp_open_richardson}
We have $(\Rcell{v}{w})^\perp = \Rcell{ww_0}{vw_0}$ for all $v \le w$ in $W$.
\end{lemma}

\begin{proof}
By \eqref{perp_group} and \eqref{inverse_permutation_matrix}, we have
\[
(\Scell{w})^\perp = \transpose{((B_+\dot{w})^{-1})}\dot{w}_0\cdot B_+ = \transpose{((B_+)^{-1})}\transpose{(\dot{w}^{-1})}\dot{w}_0\cdot B_+ = B_-\dot{w}\dot{w_0}\cdot B_+ = \Scellop{ww_0}.
\]
Similarly, we have $(\Scellop{v})^\perp = \Scell{vw_0}$. The result then follows from the definition \eqref{richardson_defn}.
\end{proof}

We point out that right multiplication by $w_0$ is an anti-automorphism of $W$ (i.e.\ $v \le w$ if and only if $ww_0 \le vw_0$). Motivated by this and \cref{perp_open_richardson}, we define
\[
[v,w]^\perp \coloneqq [ww_0,vw_0] \quad \text{ for $v\le w$ in $W$},
\]
so that $\cdot^\perp$ is an involution on the set of Bruhat intervals of $W$.

\begin{example}\label{eg:richardson}
Let $G = \SL_3(\CC)$, as in \cref{eg:flag_variety}. Then
\begin{align*}
\Rcell{132}{312}
=
\Scellop{132} \cap \Scell{312}
&=
\left\{\begin{bmatrix}
1 & 0 & 0 \\
* & 0 & -1 \\
* & 1 & 0
\end{bmatrix}\cdot B_+\right\}
\cap
\left\{\begin{bmatrix}
* & -1 & 0 \\
* & 0 & -1 \\
1 & 0 & 0
\end{bmatrix}\cdot B_+\right\} \\[2pt]
&=
\left\{\begin{bmatrix}
1 & 0 & 0 \\
0 & 0 & -1 \\
a & 1 & 0
\end{bmatrix}\cdot B_+ \;\middle\vert\; a\in\CC^\times\right\},
\end{align*}
where each $*$ denotes an arbitrary element of $\CC$. Then applying $\cdot^\perp$ (using \eqref{perp_group}) gives
\[
(\Rcell{132}{312})^\perp = \left\{\begin{bmatrix}
0 & a & 1 \\
-1 & 0 & 0 \\
0 & -1 & 0
\end{bmatrix}\cdot B_+ \;\middle\vert\; a\in\CC^\times\right\} = \Rcell{213}{231},
\]
in agreement with \cref{perp_open_richardson}.
\end{example}

\subsection{Partial flag varieties}\label{sec:partial}
For $J\subseteq I$, let $W_J\subseteq W$ be the subgroup generated by $\{s_i \mid i\in J\}$. Define $\Phi^J$ to be the root subsystem of $\Phi$ corresponding to $W_J$, and $\Phi^J_+$ to be its subset of positive roots. The \newword{standard parabolic subgroup} $P_+^J$ is the subgroup of $G$ generated by $B_+$ and $\{\dot{w} \mid w\in W_J\}$. We define $\cP_J$ to be the collection of parabolic subgroups conjugate to $P_+^J$, called a \newword{partial flag variety}. We may identify $\cP_J$ with the quotient $G/P_+^J$ via $g\cdot P_+^J \leftrightarrow gP_+^J$. Note that $\cP_I = \{G\}$ and $\cP_\varnothing = \cB$.

Given a parabolic subgroup $P\subseteq G$, there exists a unique $J\subseteq I$ such that $P\in \cP_J$; the set $J$ is called the \newword{type} of $P$. Also, given a Borel subgroup $B\in \cB$ and $J\subseteq I$, there exists a unique parabolic subgroup of type $J$ containing $B$ (see \cite[Section 23.1]{humphreys75}), which we denote by $\rho_J(B)$. Explicitly, if $B = g\cdot B_+$, then $\rho_J(B) = g\cdot P_+^J$.

\begin{example}\label{eg:partial_flag_variety}
Let $G = \SL_n(\CC)$, as in \cref{eg:flag_variety}. Take $J \subseteq I = [n-1]$, and write $J = [n-1]\setminus\{k_1 < \dots < k_l\}$. Then proceeding as in \cref{eg:flag_variety}, we can identify $\cP_J$ with the set of partial flags
\[
F_\bullet = (0 \subset F_1 \subset \cdots \subset F_l \subset \CC^n),
\]
where each $F_i$ is a $k_i$-dimensional subspace of $\CC^n$. Explicitly, $F_\bullet$ is identified with its stabilizer subgroup.

As a special case, if $J = [n-1]\setminus\{k\}$, then we can identify $\cP_J$ with the set of all $k$-dimensional subspaces of $\CC^n$, called the \newword{Grassmannian} $\Gr_{k,n}$. We recall the \emph{Pl\"{u}cker embedding} of $\Gr_{k,n}$ inside $\mathbb{CP}^{\binom{n}{k}-1}$. Let $\binom{[n]}{k}$ denote the set of $k$-element subsets of $[n]$. Given $V\in\Gr_{k,n}$, take an $n\times k$ matrix $A$ whose columns form a basis for $V$. Then the \newword{Pl\"{u}cker coordinate} $\Delta_I(V)$ (for $I\in\binom{[n]}{k}$) is defined to be the $k\times k$ minor of $A$ located in the rows $I$. The Pl\"{u}cker embedding sends $V$ to $(\Delta_I(V))_{I\in\binom{[n]}{k}}\in\mathbb{CP}^{\binom{n}{k}}$. In particular, the Pl\"{u}cker coordinates of $V$ are only defined up to global rescaling. For example, if $V\in\
\Gr_{2,4}$ is the column span of the matrix
\begin{align}\label{example_Gr24}
A = \begin{bmatrix}
1 & 0 \\
0 & 1 \\
-2 & 3 \\
-4 & 5
\end{bmatrix},
\quad\text{ then }\quad\;
\begin{aligned}
\Delta_{1,2}(V) &= 1,\;& \Delta_{2,3}(V) &= 2,\\
\Delta_{1,3}(V) &= 3,\;& \Delta_{2,4}(V) &= 4,\\
\Delta_{1,4}(V) &= 5,\;& \Delta_{3,4}(V) &= 2.
\end{aligned}\\[-36pt]
\nonumber
\end{align}
\end{example}

\subsection{Total positivity and total negativity}\label{sec:positivity}
Following Lusztig's paper \cite{lusztig94}, we introduce the \newword{totally nonnegative} and \newword{totally positive parts} of various spaces introduced above, which we denote by adding `$\ge\!0$' and `$>\!0$', respectively, as a superscript or subscript. In all cases the totally positive part will be the interior of the totally nonnegative part.

We define $U_+^{\ge 0}$ to be the submonoid of $U_+$ generated by $x_i(t)$ for all $i\in I$ and $t\in\RR_{>0}$, and we define
\[
U_+^{>0} \coloneqq \{x_{i_1}(t_1)\cdots x_{i_l}(t_l) \mid t_1, \dots, t_l\in\RR_{>0}\},
\]
for any reduced expression $w_0 = s_{i_1}\cdots s_{i_l}$. We similarly define $U_-^{\ge 0}$ and $U_-^{>0}$ by replacing $x_i$ with $y_i$, that is, $U_-^{\ge 0} = \transpose{(U_+^{\ge 0})}$ and $U_-^{>0} = \transpose{(U_+^{>0})}$.

We define $(T_0)_{\ge 0} = (T_0)_{>0}$ to be the submonoid of $T_0$ generated by $\alpha^\vee_i(t)$ for all $i\in I$ and $t\in\RR_{>0}$. We define $G_{\ge 0}$ to be the submonoid of $G$ generated by $U_+^{\ge 0}$, $U_-^{\ge 0}$, and $(T_0)_{>0}$, and we define
\[
G_{>0} \coloneqq U_-^{>0} (T_0)_{>0} U_+^{>0} = U_+^{>0} (T_0)_{>0} U_-^{>0}.
\]

For $J\subseteq I$, we define $\cP_J^{>0} \coloneqq \{u\cdot P_+^J \mid u\in U_-^{>0}\}$, and let $\cP_J^{\ge 0}$ be the Euclidean closure of $\cP_J^{>0}$. In particular, this defines $\cB_{>0}$ and $\cB_{\ge 0}$ when we take $J = \varnothing$.

The space $\cP_J^{\ge 0}$ has a cell decomposition, which we describe explicitly in the case $\cP_J = \cB$. For $v,w\in W$, we define $\Rtp{v}{w} \coloneqq \Rcell{v}{w} \cap \cB_{\geq 0}$. Rietsch \cite{rietsch99} showed that $\Rtp{v}{w} \cong \RR_{>0}^{\ell(w) - \ell(v)}$ for all $v \le w$, i.e., $\Rtp{v}{w}$ is an open cell of dimension $\ell(w) - \ell(v)$. We have the cell decomposition
\begin{align}\label{decomposition_flag}
\cB_{\ge 0} = \bigsqcup_{v \le w}\Rtp{v}{w} \cong \bigsqcup_{v\le w}\RR_{>0}^{\ell(w) - \ell(v)}.
\end{align}
In fact, $\cB_{\ge 0}$ is a regular CW complex homeomorphic to a closed ball \cite{galashin_karp_lam22}. Rietsch \cite[Theorem 4.1]{rietsch06b} determined the closure relations (in the Euclidean topology):
\begin{align}\label{richardson_closure_tnn}
\overline{\Rtp{v}{w}} = \bigsqcup_{\substack{v'\le w',\\ v \le v' \le w' \le w}}\Rtp{v'}{w'} \quad \text{ for all } v\le w.
\end{align}
We point out that since $\Rcell{w}{w} = \{\dot{w}\cdot B_+\}$, we have $\dot{w}\cdot B_+\in\cB_{\ge 0}$ for all $w\in W$.

We similarly define the \newword{totally nonpositive} and \newword{totally negative parts} by replacing the pinning of $G$ with the one obtained by inverting $x_i$ and $y_i$ for all $i\in I$. Explicitly, for $X = U_+, U_-, T_0, G$ we define
\[
X_{\le 0} \coloneqq (X_{\ge 0})^{-1} \quad \text{ and } \quad X_{<0} \coloneqq (X_{>0})^{-1}.
\]
(Note that the totally negative part of $T_0$ is the same as the totally positive part.) For $J\subseteq I$, we define $\cP_J^{<0} \coloneqq \{u\cdot P_+^J \mid u\in U_-^{<0}\}$, and we let $\cP_J^{\le 0}$ be the Euclidean closure of $\cP_J^{<0}$. For $v,w\in W$, we define $\Rtn{v}{w} \coloneqq \Rcell{v}{w} \cap \cB_{\leq 0}$, which is an open cell of dimension $\ell(w) - \ell(v)$ for all $v\le w$.

\begin{example}\label{eg:positivity}
We continue the setup of \cref{eg:flag_variety}, where $G=\SL_n(\CC)$. An element of $G$ is totally positive if and only if all of its minors are positive (respectively, nonnegative); this is the classical `Loewner--Whitney theorem' (see \cite[Theorem 12]{fomin_zelevinsky00}). Similarly, an element of $G$ is totally nonnegative if and only if all of its minors are nonnegative.

Recall that we are identifying $\cB$ with $\Fl_n$ (the collection of complete flags in $\CC^n$), and that we defined Pl\"{u}cker coordinates on $\Gr_{k,n}$ in \cref{eg:partial_flag_variety}. A complete flag $F_\bullet$ is totally positive (respectively, totally nonnegative) if and only if for all $k \in [n-1]$, the Pl\"{u}cker coordinates of $F_k\in\Gr_{k,n}$ are all positive (respectively, nonnegative) up to global rescaling; see \cite[Theorem 1.1]{bloch_karp23} or \cite[Theorem 5.27]{boretsky}. Equivalently, a Borel subgroup $g\cdot B_+$ (for $g\in\SL_n(\CC)$) is totally positive (respectively, totally nonnegative) if and only if for all $k\in [n-1]$, the ratio of any two nonzero left-justified $k\times k$ minors of $g$ is positive (respectively, nonnegative).

For a concrete example, a complete flag $F_\bullet\in\Fl_3$ is totally nonnegative if and only if \begin{align}\label{tnn_explicit_Fl3}
\begin{gathered}
\text{the Pl\"{u}cker vectors $(\Delta_1(F_1) : \Delta_2(F_1) : \Delta_3(F_1))$ and} \\
\text{$(\Delta_{1,2}(F_2) : \Delta_{1,3}(F_3) : \Delta_{2,3}(F_3))$ are both nonnegative (up to rescaling)}.
\end{gathered}
\end{align}
We can use \eqref{tnn_explicit_Fl3} to verify that the flag $F_\bullet$ from \cref{eg:flag_variety} is totally nonnegative: by \eqref{plucker_vector1}, $F_1$ has Pl\"{u}cker vector $(1 : 2 : 3)$ and $F_2$ has Pl\"{u}cker vector $(1 : 4 : 5)$, which are both nonnegative (up to rescaling). In fact, since neither Pl\"{u}cker vector has zero entries, $F_\bullet$ is totally positive.

As another example, recall from \cref{eg:partial_flag_variety} that we can identify the Grassmannian $\Gr_{k,n}$ with the partial flag variety $\cP_J$ for $J = [n-1]\setminus \{k\}$. Then an element of $\Gr_{k,n}$ is totally nonnegative (respectively, totally positive) if and only if all of its Pl\"{u}cker coordinates are nonnegative (respectively, positive), up to rescaling. (For example, the element $V\in\Gr_{2,4}$ from \eqref{example_Gr24} is totally positive.) This result is due to Rietsch; see \cite[Section 1.4]{bloch_karp23} for further discussion, and see \cite{bloch_karp23,barkley_boretsky_eur_gao} for extensions to other partial flag varieties.

We now discuss total negativity for $\SL_n(\CC)$. Given $g\in\SL_n(\CC)$, for $I,J\subseteq [n]$ with $|I| = |J|$, we let $\Delta_{I,J}(g)$ denote the minor of $g$ located in rows $I$ and columns $J$. By Jacobi's formula, we have
\[
(-1)^{\sumof{I} + \sumof{J}}\Delta_{I,J}(g) = \Delta_{[n]\setminus J, [n]\setminus I}(g^{-1}),
\]
where $\sumof{I}$ denotes the sum of the elements of $I$. In particular, $g$ is totally negative (respectively, totally nonpositive) if and only if $(-1)^{\sumof{I} + \sumof{J}}\Delta_{I,J}(g)$ is positive (respectively, nonnegative) for all $I,J\subseteq [n]$ with $|I| = |J|$. Similarly, a complete flag $F_\bullet\in\Fl_n$ is totally negative (respectively, totally nonpositive) if and only if for all $k\in [n-1]$, we can rescale the Pl\"{u}cker coordinates of $F_k$ so that $(-1)^{\sumof{I}}\Delta_I(F_k) > 0$ (respectively, $\ge 0$) for all $I\in\binom{[n]}{k}$.

For a concrete example, a complete flag $F_\bullet\in\Fl_3$ is totally nonpositive if and only if \begin{align}\label{tnp_explicit_Fl3}
\begin{gathered}
\text{the signed Pl\"{u}cker vectors $(\Delta_1(F_1) : -\Delta_2(F_1) : \Delta_3(F_1))$ and} \\
\text{$(-\Delta_{1,2}(F_2) : \Delta_{1,3}(F_3) : -\Delta_{2,3}(F_3))$ are both nonnegative (up to rescaling)}.
\end{gathered}
\end{align}
We can use \eqref{tnp_explicit_Fl3} to verify directly that the flag $(F_\bullet)^\perp$ from \cref{eg:flag_variety} is totally nonpositive: by \eqref{plucker_vector2}, the corresponding signed Pl\"{u}cker vectors are $(5 : 4: 1)$ and $(-3: -2: -1)$, which are both nonnegative (up to rescaling). In fact, since neither signed Pl\"{u}cker vector has zero entries, $(F_\bullet)^\perp$ is totally negative.
\end{example}

\begin{lemma}\label{tnn_perp}
The involution $\cdot^\perp$ takes $\cB_{\ge 0}$ onto $\cB_{\le 0}$. In particular, we have
\begin{align}\label{richardson_perp}
(\Rtp{v}{w})^\perp = \Rtn{ww_0}{vw_0} \quad \text{ for all } v \le w \text{ in } W.
\end{align}
\end{lemma}

\begin{proof}
The map $\transpose{\cdot}$ preserves $G_{\ge 0}$ by definition, and $G_{\le 0} = (G_{\ge 0})^{-1}$. Therefore the map $g\mapsto \transpose{(g^{-1})}$ takes $G_{\ge 0}$ onto $G_{\le 0}$. Now $G_{\ge 0}\cdot\cB_{\ge 0} = \cB_{\ge 0}$ (see \cite[Proposition 8.12(b)]{lusztig94}), so $G_{\le 0}\cdot\cB_{\le 0} = \cB_{\le 0}$. Since $\dot{w}_0\cdot B_+\in\cB_{\le 0}$, it follows from \eqref{perp_group} that $\cdot^\perp$ takes $\cB_{\ge 0}$ inside $\cB_{\le 0}$. Similarly, $\cdot^\perp$ takes $\cB_{\le 0}$ inside $\cB_{\ge 0}$. Since $\cdot^\perp$ is an involution, it therefore takes $\cB_{\ge 0}$ onto $\cB_{\le 0}$. This implies \eqref{richardson_perp} in light of \cref{perp_open_richardson}.
\end{proof}

\subsection{Opposition for Borel subgroups}\label{sec:opposition_borel}
We say that two Borel subgroups $B,B'\in\cB$ are \newword{opposed} if their intersection $B\cap B'$ is a maximal torus. We have the following test for opposition:
\begin{lemma}\label{opposition_Gaussian}
Let $g,h\in G$.
\begin{enumerate}[label=(\roman*), leftmargin=*, itemsep=2pt]
\item\label{opposition_Gaussian_minus}
The elements $g\cdot B_+$ and $h\cdot B_-$ of $\cB$ are opposed if and only if $h^{-1}g\in B_-B_+$.
\item\label{opposition_Gaussian_plus}
The elements $g\cdot B_+$ and $(h\cdot B_+)^\perp$ of $\cB$ are opposed if and only if $\transpose{h}g\in B_-B_+$.
\end{enumerate}
\end{lemma}

\begin{proof}
\ref{opposition_Gaussian_minus} After conjugating by $h^{-1}$, we may assume that $h = \dot{e}$ is the identity. Let $w\in W$ denote the opposite Schubert cell $\Scellop{w}$ containing $g\cdot B_+$, so that $g = g_1\dot{w}g_2$ for some $g_1\in B_-$ and $g_2\in B_+$. We must show that $g\cdot B_+$ and $B_-$ are opposed if and only if $w = e$. We have
\[
(g\cdot B_+) \cap B_- = (g_1\dot{w}\cdot B_+) \cap (g_1\cdot B_-) = g_1\cdot ((\dot{w}\cdot B_+) \cap B_-),
\]
and $(\dot{w}\cdot B_+)\cap B_-$ is a maximal torus if and only if $w = e$.

\ref{opposition_Gaussian_plus} This follows from \ref{opposition_Gaussian_minus} since $(h\cdot B_+)^\perp = \transpose{(h^{-1})}\dot{w}_0\cdot B_+ = \transpose{(h^{-1})}\cdot B_-$.
\end{proof}

\subsection{Spaces of maximal tori}\label{background_tori}
Following Lusztig \cite{lusztig}, we let $\cT$ denote the set of all maximal tori of $G$. Note that
\[
\cT = \{g\cdot T_0\mid g\in G\},
\]
which identifies $\cT$ with $G/N(T_0)$.

We will also find it useful to work with the space of \newword{framed maximal tori}:
\[
\framed \coloneqq \{(T,B)\in\cT\times\cB \mid T\subseteq B\}.
\]
(Such pairs $(T,B)$ are also known as \emph{Borel pairs} in the literature.) We can similarly identify $\framed$ with $G/T_0$. We will make frequent use of the algebraic map
\[
\oppborel: \framed \to \cB, \quad (T,B) \mapsto B',
\]
where $B'$ is the unique Borel subgroup opposed to $B$ and containing $T$.

When $G = \SL_n(\CC)$, we can identify $\cT$ with the set of (unordered) bases of $\CC^n$ modulo rescaling of the basis vectors, while $\framed$ is the set of ordered bases modulo rescaling. In this setting, the map $(T,B) \mapsto (T,\oppborel(T,B))$ on $\cT$ just reverses the order of the basis vectors.

We define the space of \newword{totally positive maximal tori} by
\[
\cT_{>0} \coloneqq \{B \cap B'\mid B\in\cB_{>0} \text{ and } B'\in\cB_{<0}\}.
\]
This is well-defined since every $B\in\cB_{>0}$ is opposed to every $B'\in\cB_{<0}$, as shown by Lusztig (see \cref{tnn_tp_opposed} for a generalization). Moreover, we can recover $B$ and $B'$ from $B\cap B'$:
\begin{proposition}[{Lusztig \cite[Proposition 1.3]{lusztig}}]\label{tp_torus_parametrization}
The map
\begin{align}\label{torus_bijection}
\cB_{>0} \times \cB_{<0} \to \cT_{>0}, \quad (B,B') \mapsto B\cap B'
\end{align}
is a bijection.\hfill\qed
\end{proposition}

We define the space of \newword{totally nonnegative maximal tori}, denoted $\cT_{\ge 0}$, to be the Euclidean closure of $\cT_{>0}$ inside $\cT$. We have the following alternative description of $\cT_{\ge 0}$:
\begin{proposition}\label{closure_equality}
We have
\begin{align}\label{tori_closure}
\cT_{\ge 0} = \{B\cap B'\mid (B,B') \in \cB_{\ge 0}\times\cB_{\le 0} \text{ such that $B$ and $B'$ are opposed}\}.
\end{align}
\end{proposition}

\begin{proof}
($\supseteq$) This containment follows by taking limits, since $\cB_{\ge 0}$ is the Euclidean closure of $\cB_{>0}$ and $\cB_{\le 0}$ is the Euclidean closure of $\cB_{<0}$.

($\subseteq$) Given $T\in\cT_{\ge 0}$, we must write $T$ as an intersection as in \eqref{tori_closure}. Write $T$ as a limit of the sequence $(T_j)_{j\ge 0}$ in $\cT_{>0}$, and for each $j\ge 0$ write $T_j = B_j \cap B'_j$ for some $(B_j,B'_j) \in \cB_{>0} \times \cB_{<0}$. Since $\cB_{\ge 0}$ is compact, by passing to a subsequence we may assume that $(B_j)_{j\ge 0}$ converges to some $B\in\cB_{\ge 0}$, which necessarily contains $T$. Since $B'_j = \oppborel(T_j,B_j)$ for all $j\ge 0$ and $\oppborel$ is continuous, we have that $(B'_j)_{j\ge 0}$ converges to $B' \coloneqq \oppborel(T,B)$. Therefore $B'\in\cB_{\le 0}$, and since $T = B\cap B'$ we are done.
\end{proof}

\begin{example}\label{eg:closure_equality}
We point out that there is not always a unique way to write a totally nonnegative torus as an intersection as in \eqref{tori_closure}. For example, we can write the standard torus as
\[
T_0 = \dot{w}\cdot B_+ \cap \dot{w}\cdot B_- \quad \text{ for all } w\in W,
\]
where $(\dot{w}\cdot B_+, \dot{w}\cdot B_-) \in \cB_{\ge 0}\times\cB_{\le 0}$. We can obtain uniqueness by adding certain additional assumptions to \eqref{tori_closure}; see \cref{tnn_tp_torus_parametrization,framed_tori_unique}.
\end{example}

\subsection{Positive weight bases}\label{sec:positive_bases}
Let $\mathfrak g$ denote the Lie algebra of $G$, with Chevalley generators $e_i$, $f_i$, and $h_i \coloneqq [e_i,f_i]$ for $i\in I$, where $x_i(1) = \exp(e_i)$ and $y_i(1) = \exp(f_i)$. Given a dominant weight $\lambda\in X(T_0)$, let $V_\lambda$ be the corresponding irreducible representation of $G$. We say that a basis $(b_x)_{x\in X}$ of $V_\lambda$ is a \newword{positive weight basis} if it is a weight basis such that for all $i\in I$ and $y\in X$, the vector $f_ib_y$ has nonnegative coefficients when expanded in the basis $(b_x)_{x\in X}$. We will study the properties of positive weight bases in \cref{sec:opposition_combinatorial}.

When $G$ is simply laced, Lusztig's \emph{canonical basis} \cite{lusztig90} and \emph{semicanonical basis} \cite{lusztig92} are both positive weight bases of $V_\lambda$. Kashiwara \cite{kashiwara91} generalized the canonical basis to all $G$, but the resulting weight basis (the \emph{upper global basis}) is not necessarily positive when $G$ is not simply-laced \cite{tsuchioka10}. That said, one can often deduce results for general $G$ from the case of simply-laced $G$ using the technique of \emph{folding} (see, e.g., \cite[Section 1.5]{lusztig94}). Instead, we will employ the \newword{MV basis} due to Mirkovi\'{c} and Vilonen \cite{mirkovic_vilonen07}, which gives rise to a positive weight basis of $V_\lambda$ by work of Baumann, Kamnitzer, and Knutson \cite{baumann_kamnitzer_knutson21}. (See Kamnitzer's survey \cite{kamnitzer23} for further discussion of these various bases.)
\begin{theorem}[{Baumann, Kamnitzer, and Knutson \cite{baumann_kamnitzer_knutson21}}]\label{positive_basis_exists}
Let $\lambda$ be a dominant weight. Then $V_\lambda$ has a positive weight basis.\hfill\qed
\end{theorem}

\begin{proof}
By \cite[Theorem 1.3]{baumann_kamnitzer_knutson21}, the MV basis gives a ``biperfect'' basis for $\CC[U_+]$ with the property that $e_i$ acts with nonnegative structure constants for all $i\in I$. Taking the Lie-theoretic transpose \eqref{lie_transpose} gives a biperfect basis for $\CC[U_-]$ such that $f_i$ acts with nonnegative structure constants for all $i\in I$. Taking the linear-algebraic transpose gives a basis $X$ for the universal enveloping algebra $U(\mathfrak{n}_-)$ (where $\mathfrak{n}_-$ is the Lie algebra of $U_-$) such that $f_i$ acts with nonnegative structure constants. The biperfect property implies that, for each dominant weight $\lambda$, there is a subset $X_\lambda\subseteq X$ such the canonical map $U(\mathfrak{n}_-) \to V_\lambda$ sends $X_\lambda$ to a basis for $V_\lambda$ and sends every basis vector in $X\setminus X_\lambda$ to zero. Then $X_\lambda$ is a positive weight basis of $V_\lambda$.
\end{proof}

\subsection{Generalized minors}
We introduce generalized minors, following \cite[Section 2]{fomin_zelevinsky00a} and \cite[Definition 6.2]{marsh_rietsch04}. Let $(\varpi_i)_{i\in I}$ denote the \newword{fundamental weights}, which form the basis of $X(T_0)$ dual to the simple coroots $(\alpha_i^\vee)_{i\in I}$. Let $V_{\varpi_i}$ be the irreducible representation of $G$ with highest weight $\varpi_i$ and highest weight vector $\xi_{\varpi_i}$. An \newword{extremal weight vector} of $V_{\varpi_i}$ is a vector of the form $\dot w \xi_{\varpi_i}$, for some $w\in W$. A \newword{generalized minor} (for $\varpi_i$) is an element of $\CC[G]$ of the form
\[ g \mapsto \langle \eta_2, g\eta_1\rangle, \]
where $\eta_1,\eta_2$ are extremal weight vectors of $V_{\varpi_i}$, and the notation $\langle \eta_2,g\eta_1\rangle$ indicates to expand $g\eta_1$ in a weight basis of $V_{\varpi_i}$ containing $\eta_2$, and then take the coefficient of $\eta_2$ (this does not depend on the choice of weight basis). When $G = \SL_n(\CC)$, the generalized minors for $\varpi_i$ are precisely the $i\times i$ matrix minors.

We will need the following connection between generalized minors and total positivity:
\begin{theorem}[{Fomin and Zelevinsky \cite[Theorems 3.1 and 3.2]{fomin_zelevinsky00a}}]\label{fominzelevinsky}
An element of $G$ is totally positive if and only if all of its generalized minors are positive.\hfill\qed
\end{theorem}

\section{Warmup: opposition for Grassmannians}\label{sec:grassmannians}

\noindent In this section we study the notion of opposition between two linear subspaces of $\CC^n$, which we call \emph{transversality}. This will serve as a concrete motivation and warmup for some of our later results and technical arguments for opposition on more general partial flag varieties. It will also provide the foundation for our characterization of opposition for Bruhat intervals in type $A$ (see \cref{opposition_A}).

Recall from \cref{eg:partial_flag_variety} that $\Gr_{k,n}$ denotes the Grassmannian of all $k$-dimensional subspaces of $\CC^n$, which we can identify with the partial flag variety $\cP_{[n-1]\setminus\{k\}}$ when $G = \SL_n(\CC)$. Also recall that $\Delta_I(\cdot)$ denotes the Pl\"{u}cker coordinate indexed by $I\in\binom{[n]}{k}$.

Given $V\in\Gr_{k,n}$ and $W\in\Gr_{l,n}$, we say that $V$ and $W$ are \newword{transverse} if
\[
\dim(V\cap W) = \max(0, k+l-n),
\]
i.e., $V\cap W$ is as small as possible. Since $\dim(V) + \dim(W) = \dim(V+W) + \dim(V\cap W)$, this is equivalent to
\[
\dim(V+W) = \min(n,k+l),
\]
i.e., $V+W$ being as large as possible. We will show later (see \cref{prop:transversality_opposition}) that the notion of transversality above coincides with an appropriate notion of opposition for the corresponding parabolic subgroups.

We will focus on transversality for subspaces of complementary dimension, i.e., when $k+l=n$. Note that in this case, $V$ and $W$ are transverse if and only if $V\oplus W = \CC^n$.

Given $V\in\Gr_{k,n}$, recall that $V^\perp\in\Gr_{n-k,n}$ denotes the orthogonal complement of $V$. The Pl\"{u}cker coordinates of $V$ and $V^\perp$ are related as follows (see, e.g., \cite[Lemma 1.11(ii)]{karp17}):
\begin{lemma}\label{complementary_pluckers}
Let $V\in\Gr_{k,n}$. Then
\[
\Delta_{[n]\setminus I}(V^\perp) = (-1)^{\sumof{I}}\Delta_I(V) \quad\text{ for all } I\in\textstyle\binom{[n]}{k},
\]
where $\sumof{I}$ denotes the sum of the elements of $I$.\hfill\qed
\end{lemma}

We have the following characterization of transversality for subspaces of complementary dimension:
\begin{lemma}\label{plucker_opposition}
Let $V,W\in\Gr_{k,n}$. Then $V$ and $W^\perp$ are transverse if and only if
\begin{align}\label{plucker_opposition_sum}
\sum_{I\in\binom{[n]}{k}}\Delta_I(V)\Delta_I(W) \neq 0.
\end{align}
\end{lemma}

\begin{proof}
Let $A$ be an $n\times k$ matrix whose columns form a basis for $V$, let $B$ be an $n\times (n-k)$ matrix whose columns form a basis for $W^\perp$, and let $C = [A \;|\; B]$ be the concatenation of $A$ and $B$ (which is an $n\times n$ matrix). Then we have the following chain of equivalent statements:
\begin{align*}
&\mathrel{\hphantom{\Leftrightarrow}\,}\text{$V$ and $W^\perp$ are transverse} \\
&\Leftrightarrow\, V\oplus W^\perp = \CC^n \\
&\Leftrightarrow\, \text{the columns of $C$ form a basis for $\CC^n$} \\
&\Leftrightarrow\, \det(C) \neq 0 \\
&\Leftrightarrow\, \sum_{I\in\binom{[n]}{k}}(-1)^{\sumof{I}}\Delta_I(V)\Delta_{[n]\setminus I}(W^\perp) \neq 0 \quad \text{(by Laplace expansion)},
\end{align*}
which is equivalent to \eqref{plucker_opposition_sum} applying by \cref{complementary_pluckers} to $W$.
\end{proof}

Now we additionally consider positivity and negativity. Recall from \cref{eg:positivity} that an element $V\in\Gr_{k,n}$ is totally nonnegative (denoted $V\in\Gr_{k,n}^{\ge 0}$) if and only if all the Pl\"{u}cker coordinates of $V$ are nonnegative (up to rescaling). Also, we write $V\in\Gr_{k,n}^{\le 0}$ if $V$ is totally nonpositive. By \cref{complementary_pluckers}, we have
\begin{align}\label{tnp_grassmannian}
V\in\Gr_{k,n}^{\le 0}
\quad\Leftrightarrow\quad
V^\perp\in\Gr_{n-k,n}^{\ge 0}
\quad\Leftrightarrow\quad
(-1)^{\sumof{I}}\Delta_I(V) \ge 0 \text{ for all } I\in\textstyle\binom{[n]}{k}.
\end{align}

We also recall Postnikov's cell decomposition of the totally nonnegative Grassmannian $\Gr_{k,n}^{\ge 0}$ from \cite{postnikov06}. Given $V\in\Gr_{k,n}^{\ge 0}$, we define the \newword{positroid} of $V$ to be
\[
M \coloneqq \{I\in\textstyle\binom{[n]}{k} \mid \Delta_I(V) \neq 0\}.
\]
Conversely, given a collection $M$ of $k$-element subsets of $[n]$, we define
\[
S_M \coloneqq \{V\in\Gr_{k,n}^{\ge 0} \mid \text{$M$ is the positroid of $V$}\},
\]
which is called a \newword{positroid cell} if $V$ is nonempty. This provides the cell decomposition
\begin{align}\label{positroid_decomposition}
\Gr_{k,n}^{\ge 0} = \bigsqcup_M S_M,
\end{align}
where the disjoint union is over all $M$ such that $S_M\neq\varnothing$.

\begin{proposition}\label{positroid_opposition}
Let $V$ and $W$ be totally nonnegative elements of $\Gr_{k,n}$, and let $M$ and $N$ be their respective positroids. Then $V$ and $W^\perp$ are transverse if and only if $M$ and $N$ intersect. In particular, whether $V$ and $W^\perp$ are transverse depends only on the pair of positroid cells containing $V$ and $W$.
\end{proposition}

\begin{proof}
By \cref{plucker_opposition}, $V$ and $W^\perp$ are transverse if and only if \eqref{plucker_opposition_sum} holds. Since $V$ and $W$ are totally nonnegative, every term in the sum on the left-hand side of \eqref{plucker_opposition_sum} is nonnegative. Thus the sum is nonzero if and only if some term is nonzero, i.e., if there exists $I\in M\cap N$.
\end{proof}

\cref{positroid_opposition} gives a combinatorial description of opposition for two subspaces of complementary dimensions, one of which is totally nonnegative and the other totally nonpositive. We leave it as an open problem to generalize this to arbitrary dimensions:
\begin{problem}
Let $V\in\Gr_{k,n}^{\ge 0}$ and $W\in\Gr_{l,n}^{\ge 0}$. Find a combinatorial condition for when $V$ and $W^\perp$ are opposed.
\end{problem}

\section{Opposition for parabolic subgroups}\label{sec:opposition_parabolic}

\noindent In this section we generalize the notion of opposition for Borel subgroups (from \cref{sec:opposition_borel}) to parabolic subgroups, and prove various properties about it. We also give a useful criterion for testing opposition of Borel subgroups via maximal parabolic subgroups (see \cref{opposition_maximal_parabolics}). We will use this criterion in \cref{sec:opposition_A} to characterize opposed Bruhat intervals in type $A$.

\subsection{Definition of opposition}
Fix $J,J'\subseteq I$. The group $G$ acts transitively on $\cP_{J}\times \cP_{J'}$, and each $G$-orbit contains at least one pair of the form $(P_+^{J},\dot{w}\cdot P_+^{J'})$ for some $w\in W$. The pairs $(P_+^{J},\dot{w}\cdot P_+^{J'})$ and $(P_+^{J},\dot{w}'\cdot P_+^{J'})$ are in the same $G$-orbit if and only if $W_{J}wW_{J'} = W_{J}w'W_{J'}$.

Given parabolic subgroups $P\in \cP_{J}$ and $P'\in\cP_{J'}$, the \newword{relative position} of the pair $(P, P')$ is the unique double coset $W_{J}wW_{J'}$ of $W$ such that $(P,P')$ is in the $G$-orbit of $(P_+^{J},\dot{w}\cdot P_+^{J'})$. Note that if the relative position of $(P, P')$ is $W_{J}wW_{J'}$, then the relative position of $(P', P)$ is $W_{J'}w^{-1}W_{J}$. We say that $P$ and $P'$ are \newword{opposed} if the relative position of $(P, P')$ is $W_{J}w_0W_{J'}$. Note that this is equivalent to the relative position of $(P',P)$ being $W_{J'}w_0W_J$, so being opposed is a symmetric relation.

If $J = J' = \varnothing$, our notion of opposition recovers the notion of opposition for Borel subgroups; this follows from \cref{opposition_Gaussian}\ref{opposition_Gaussian_minus}.
If $J' = J^*$ (where $\cdot^*$ is defined as in \eqref{star_involution}), then \cref{prop:opposeddualparabolics} (proved below) implies that $P$ and $P'$ are opposed if and only if $P\cap P'$ is a Levi subgroup of $P$ (equivalently, of $P'$). This recovers the definition of opposition when $J' = J^*$ of He \cite[Section 1.4]{he04}.

\subsection{Relation to transversality}\label{sec:transversality_opposition}

In \cref{sec:grassmannians}, we discussed what it means for subspaces $V\in\Gr_{k,n}$ and $W\in \Gr_{l,n}$ to be transverse. Here we will show that this is the same as the notion of opposition for the corresponding parabolic subgroups. Recall from \cref{eg:partial_flag_variety} that $\Stab : \Gr_{k,n} \to \cP_{I\setminus \{k\}}$ is an isomorphism, where $\Stab(V)$ to denotes the subgroup of $\SL_n(\CC)$ stabilizing $V\in \Gr_{k,n}$. The inverse map sends $g\cdot P_+^{I\setminus \{k\}}$ to the span of the first $k$ columns of the matrix $g$.

\begin{proposition}\label{prop:transversality_opposition}
Let $V\in\Gr_{k,n}$ and $W\in\Gr_{l,n}$. Then $V$ and $W$ are transverse if and only if the parabolic subgroups $\Stab(V)$ and $\Stab(W)$ are opposed.
\end{proposition}
\begin{proof}
    The map $(\Stab, \Stab) : \Gr_{k,n} \times \Gr_{l,n} \to \cP^{I\setminus \{k\}} \times \cP^{I\setminus \{l\}}$
    is an isomorphism, which is equivariant with respect to the diagonal $G$ action. Therefore $\Stab(V)$ and $\Stab(W)$ are opposed if and only if there exists $g\in G$ so that $g\cdot V$ is the span of $e_1,\ldots, e_k$ and $g\cdot W$ is the span of $e_{n-l+1}, \ldots, e_{n}$ (where $e_1, \dots, e_n$ denote the standard basis vectors of $\CC^n$). Hence if $\Stab(V)$ and $\Stab(W)$ are opposed, then $V$ and $W$ are transverse.
    
    For the converse, we must show that if $V$ and $W$ are transverse, there exists $g\in G$ so that $g\cdot V$ is the span of $e_1,\ldots,e_k$ and $g\cdot W$ is the span of $e_{n-l+1},\ldots,e_n$. To this end, by linear algebra we may take $X_1\subseteq X_2 \subseteq X_3 \subseteq X_4 \subseteq \CC^n$ so that
    \begin{itemize}[itemsep=2pt]
        \item $X_1$ is a basis for $V\cap W$;
        \item $X_2$ is a basis for $V$;
        \item $X_3$ is a basis for $V+W$ (so that $X_1 \cup (X_3\setminus X_2)$ is a basis for $W$); and
        \item $X_4$ is a basis for $\CC^n$.
    \end{itemize}
    If $V$ and $W$ are transverse, then $\dim(V+W) = \min(n,k+l)$, so
    \[
    |X_3\setminus X_2| = \min(n-k,l) = |\{e_{n-l+1},\ldots, e_n\}\setminus \{e_1,\ldots, e_k\}|.
    \]
    Hence we can take $g\in \GL_n(\CC)$ which sends $X_4$ to $\{e_1, \dots, e_n\}$, so that
    \begin{itemize}[itemsep=2pt]
        \item $X_1$ is sent to $\{e_{k-|X_1|+1}, \dots, e_{k-1}, e_k\}$;
        \item $X_2$ is sent to $\{e_1,\ldots, e_k\}$; and
        \item $X_3\setminus X_2$ is sent to $\{e_{n-l+1},\ldots, e_n\}\setminus \{e_1,\ldots, e_k\}$.
    \end{itemize}
    Rescaling $g$ so that it lies in $\SL_n(\CC)$, we see that $g$ has the desired properties.
\end{proof}

\subsection{Characterizations of opposition}
We will now prove various characterizations of opposition for parabolic subgroups. We need the following inequality:
\begin{lemma}\label{parabolic_inequality}
Let $P$ and $P'$ be parabolic subgroups of $G$. Then
\begin{align}\label{parabolic_inequality_equation}
\dim(P\cap P') \ge \dim(P) + \dim(P') - \dim(G),
\end{align}
with equality if and only if $P$ and $P'$ are opposed.
\end{lemma}

\begin{proof}
Let $J$ and $J'$ be the types of $P$ and $P'$, respectively. By the orbit-stabilizer theorem applied to $G$ acting on $\cP_J \times \cP_{J'}$, we have $\dim(P\cap P') = \dim(G) - \dim(\cO)$, where $\cO$ is the orbit through $(P,P')$ in $\cP_J\times \cP_{J'}$. The quantity $\dim(\cO)$ is maximized precisely when $\cO$ is the dense orbit, i.e., $P$ and $P'$ are opposed. So it is enough to show that equality holds in \eqref{parabolic_inequality_equation} when $P$ and $P'$ are opposed.

To this end, after acting by $G$, we may assume that $P=P^{J}_+$ and $P' = w_0\cdot P^{J'}_+$. Let $\mathfrak{p}$ and $\mathfrak{p}'$ denote the Lie algebras of $P$ and $P'$, respectively. Note that for a closed subgroup $H\subseteq G$ containing $T_0$ with Lie algebra $\mathfrak{h}$, we have
\begin{align}\label{dimension_formula}
\dim(H) = \dim(T_0) + |\{\alpha\in\Phi \mid \mathfrak{g}_\alpha \subseteq \mathfrak{h}\}|.
\end{align}
Since $P$ contains $B_+$ and $P'$ contains $B_-$, we have $\mathfrak{p} + \mathfrak{p'} = \mathfrak{g}$. Applying \eqref{dimension_formula} to each term in \eqref{parabolic_inequality_equation}, we see that equality holds by the inclusion-exclusion formula for finite sets.
\end{proof}

\begin{proposition}\label{prop:opposedparabolics}
    Let $P$ and $P'$ be parabolic subgroups of $G$ with Lie algebras $\mathfrak{p}$ and $\mathfrak{p}'$, respectively. Fix a maximal torus $T\subseteq P\cap P'$, which determines root subspaces $\mathfrak{g}_{\alpha} \subseteq \mathfrak{g}$ for each root $\alpha$.
    Then the following are equivalent:
    \begin{enumerate}[label=(\roman*), leftmargin=*, itemsep=2pt]
        \item\label{prop:opposedparabolics_opposed} $P$ and $P'$ are opposed;
        \item\label{prop:opposedparabolics_borels} there exist opposed Borel subgroups $B, B' \in \cB$ such that $B\subseteq P$, $B'\subseteq P'$, and $B\cap B' =T$;
        \item\label{prop:opposedparabolics_roots} for all roots subspaces $\mathfrak{g}_\alpha \subseteq \mathfrak{p}$ such that $\mathfrak{g}_{-\alpha}\not\subseteq \mathfrak{p}$, we have $\mathfrak{g}_{-\alpha}\subseteq \mathfrak{p}'$;
        \item\label{prop:opposedparabolics_roots2} every root subspace $\mathfrak{g}_\alpha$ is contained in $\mathfrak{p}$ or $\mathfrak{p}'$; and
        \item\label{prop:opposedparabolics_dimension} $\dim(P\cap P') = \dim(P) + \dim(P') - \dim(G)$.
    \end{enumerate}
\end{proposition}

\begin{proof}
    Let $J$ and $J'$ be the types of $P$ and $P'$, respectively. After acting by $G$, we may assume that $P=P^{J}_+$ and $T=T_0$.

     \ref{prop:opposedparabolics_opposed} $\Leftrightarrow$ \ref{prop:opposedparabolics_dimension}: This is just \cref{parabolic_inequality}.
    
    \ref{prop:opposedparabolics_opposed} $\Rightarrow$ \ref{prop:opposedparabolics_borels}: Suppose that $P$ and $P'$ are opposed. Since $P'$ contains $T_0$, we can write $P' = \dot{x}\cdot P_+^{J'}$ for some $x\in W$ with $W_JxW_{J'} = W_Jw_0W_{J'}$. After acting by $W_J$ (which fixes $P$ and $T_0$), we may assume that $x = w_0$. Then we may take $B = B_+$ and $B' = \dot{w}_0\cdot B_+ = B_-$.
    
    \ref{prop:opposedparabolics_borels} $\Rightarrow$ \ref{prop:opposedparabolics_roots}: Suppose that \ref{prop:opposedparabolics_borels} holds. Let $\mathfrak{b}$ and $\mathfrak{b}'$ be the Lie algebras of $B$ and $B'$, respectively. Then $\mathfrak{b}$ contains the nilpotent radical of $\mathfrak{p}$, which is exactly the sum of the root spaces $\mathfrak{g}_\alpha$ such that $\mathfrak{g}_{\alpha}\subseteq \mathfrak{p}$ and $\mathfrak{g}_{-\alpha}\not\subseteq \mathfrak p$. Hence the opposite Borel subalgebra $\mathfrak{b}'$ (and also $\mathfrak{p}'$) contains such roots spaces $\mathfrak{g}_{-\alpha}$. 

    \ref{prop:opposedparabolics_roots} $\Rightarrow$ \ref{prop:opposedparabolics_roots2}: Let $\alpha\in\Phi_+$. Then $\mathfrak{g}_\alpha \subseteq \mathfrak{p}$, and \ref{prop:opposedparabolics_roots} implies that $\mathfrak{g}_{-\alpha} \subseteq \mathfrak{p}$ or $\mathfrak{g}_{-\alpha} \subseteq \mathfrak{p}'$. This implies \ref{prop:opposedparabolics_roots2}.

    \ref{prop:opposedparabolics_roots2} $\Rightarrow$ \ref{prop:opposedparabolics_dimension}: This follows from \eqref{dimension_formula} and the inclusion-exclusion formula for finite sets.
\end{proof}

In general, it is possible for there to be a parabolic subgroup of type $J'$ which is opposed to $P_+^J$, contains $T_0$, and yet does not contain $B_-$. (For example, if $G = \SL_3(\CC)$, $I = \{1,2\}$, $J = \{1\}$, and $J' = \varnothing$, then then $\dot{s}_2\dot{s}_1\cdot B_+$ is opposed to $P_+^J$ and does not contain $B_-$.) However, if $J'=J^*$ then this does not occur. We can quantify this phenomenon by defining the \newword{excess} from $J$ to $J'$ to be
\[ \excess(J,J') \coloneqq |\Phi_+^{J'}\setminus \Phi_+^{J^*}| = |\Phi_+^{J'}| - |\Phi_+^{J'\cap J^*}| \ge 0. \]
Note that $\excess(J,J') = 0$ if and only if $J' = J^*$.
\begin{lemma}\label{lem:excess}
   Let $P\in\cP_J$ and $P'\in\cP_{J'}$ be opposed parabolic subgroups of types $J$ and $J'$, respectively, and let $U_P$ denote the unipotent radical of $P$. Then $\dim(U_P\cap P') = \excess(J,J')$.
\end{lemma}
\begin{proof}
Note that $U_{g\cdot P} = g\cdot U_P$ for all $g\in G$. Hence after acting by $G$, we may assume that $P = P^J_+$ and $P' = \dot{w}_0\cdot P^J_+$. In this case, by examining root subspaces, we see that
        \[ \dim(P'\cap U_P) = |(w_0\Phi^{J'}) \cap (\Phi_+\setminus \Phi^{J})|. \]
        By negating and applying $w_0$ to the roots in the left-hand side above, we get 
        \[ \dim(P'\cap U_P) = |\Phi^{J'} \cap (\Phi_+\setminus \Phi^{J^*})| = |\Phi^{J'}_+ \setminus \Phi^{J^*}_+| = \excess(J,J').\qedhere  \]
\end{proof}

Let $P$ and $P'$ be parabolic subgroups of $G$. We have $\dim(U_P) = \dim(G)-\dim(P)$, so by \cref{parabolic_inequality}, another equivalent criterion for $P$ and $P'$ to be opposed is that $\dim(P\cap P') = \dim(P')-\dim(U_P)$. One way this could occur is if $U_P\cap P' = \{\dot{e}\}$. But \cref{lem:excess} shows that this can only happen if the types $J$ and $J'$ of $P$ and $P'$ satisfy $J'=J^*$.
When $J'=J^*$, we have the following strengthening of \cref{prop:opposedparabolics}:
\begin{proposition}\label{prop:opposeddualparabolics}
    Let $P\in \cP_J$ and $P'\in \cP_{J^*}$ be parabolic subgroups of types $J$ and $J^*$, where $J \subseteq I$. Then the following are equivalent:
    \begin{enumerate}[label=(\roman*), leftmargin=*, itemsep=2pt]
        \item\label{prop:opposeddualparabolics_opposed} $P$ and $P'$ are opposed;
        \item\label{prop:opposeddualparabolics_opposite_torus}for every maximal torus $T\subseteq P\cap P'$ and Borel subgroup $B$ such that $T\subseteq B\subseteq P$, the opposite Borel subgroup $\oppborel(T,B)$ is contained in $P'$;

        \item\label{prop:opposeddualparabolics_radical} the intersection of $P'$ with the unipotent radical of $P$ is trivial;
        
        \item\label{prop:opposeddualparabolics_dimension} $\dim(P\cap P') = \dim(T_0) + |\Phi^J|$; and
        
        \item\label{prop:opposeddualparabolics_levi} the group $P\cap P'$ is a Levi subgroup of $P$.
    \end{enumerate}
    Moreover, if $P$ and $P'$ are opposed and $T\subseteq P\cap P'$ is a maximal torus, then $P'$ is the unique element of $\cP_{J^*}$ opposed to $P$ and containing $T$.
\end{proposition}
\begin{proof}
By \eqref{dimension_formula}, we calculate that
\begin{align}\label{parabolic_RHS}
\dim(P) + \dim(P') - \dim(G) = \dim(T_0) + |\Phi^J|.
\end{align}
    \ref{prop:opposeddualparabolics_opposed} $\Leftrightarrow$ \ref{prop:opposeddualparabolics_dimension}: This follows from \ref{prop:opposedparabolics_opposed} $\Leftrightarrow$ \ref{prop:opposedparabolics_dimension} of \cref{prop:opposedparabolics}, using \eqref{parabolic_RHS}.

    \ref{prop:opposeddualparabolics_levi} $\Rightarrow$ \ref{prop:opposeddualparabolics_dimension}: This follows from the fact that every Levi subgroup of $P$ has dimension $\dim(T_0) + |\Phi^J|$.

    \ref{prop:opposeddualparabolics_radical} $\Rightarrow$ \ref{prop:opposeddualparabolics_dimension}: We have $\dim(P\cap P') \geq \dim(T_0) + |\Phi^J|$ by \eqref{parabolic_inequality_equation} and \eqref{parabolic_RHS}. Moreover, \ref{prop:opposeddualparabolics_radical} implies that $\dim(P\cap P')\leq \dim(T_0) + |\Phi^J|$, which proves \ref{prop:opposeddualparabolics_dimension}.

    \ref{prop:opposeddualparabolics_radical} $\Rightarrow$ \ref{prop:opposeddualparabolics_levi}: If \ref{prop:opposeddualparabolics_radical} holds, then $P\cap P'$ is contained in a Levi subgroup $L\subseteq P$. We have $\dim(L) = \dim(T_0) + |\Phi^J|$, so $\dim(L) \le \dim(P\cap P')$. Hence $P\cap P' = L$, proving \ref{prop:opposeddualparabolics_levi}.
    
    \ref{prop:opposeddualparabolics_opposed} $\Rightarrow$ \ref{prop:opposeddualparabolics_radical}: This follows from \cref{lem:excess}, since $\excess(J,J^*) = 0$.

    \ref{prop:opposeddualparabolics_opposite_torus} $\Rightarrow$ \ref{prop:opposeddualparabolics_opposed}: This follows from \ref{prop:opposedparabolics_borels} $\Rightarrow$ \ref{prop:opposedparabolics_opposed} of \cref{prop:opposedparabolics}.

    \ref{prop:opposeddualparabolics_radical} $\Rightarrow$ \ref{prop:opposeddualparabolics_opposite_torus}: Given $T$ and $B$ as in \ref{prop:opposeddualparabolics_opposite_torus}, after acting by $G$ we may assume that $P=P^J_+$, $T=T_0$, and $B=B_+$. If \ref{prop:opposeddualparabolics_radical} holds, then $P' = \dot{x}\dot{w}_0\cdot P^{J^*}_+$ for some $x\in W_J$. We have $\oppborel(T,B) = B_-$, so we must show that $P'$ contains $B_-$. Since $s_i = w_0s_{i^*}w_0$, we have $\dot{s}_i \in \dot{w}_0\cdot P^{J^*}_+$ for all $i\in J$, so $\dot x \in \dot{w}_0\cdot P^{J^*}_+$. Hence
    \[
    P' = \dot{x}\dot{w}_0\cdot P^{J^*}_+ = \dot{w}_0\cdot P^{J^*}_+ \supseteq \dot{w}_0\cdot B_+ = B_-,
    \]
    which proves \ref{prop:opposeddualparabolics_opposite_torus}.

    To see the final claim, note that once we fix $T\subseteq P\cap P'$, the roots appearing in the Lie algebra $\mathfrak{p}'$ must be exactly the negations of the roots appearing in $\mathfrak{p}$. This uniquely determines $P'$.
\end{proof}

\subsection{Opposition via maximal parabolics}
Recall from \cref{sec:partial} that for $J\subseteq I$, the map $\rho_J:\cB \to \cP_J$ sends a Borel subgroup $B$ to the unique $P\in \cP_{J}$ such that $B\subseteq P$.
\begin{theorem}\label{opposition_maximal_parabolics}
    Let $B,B' \in \cB$ be Borel subgroups. Then $B$ and $B'$ are opposed if and only if for all $i\in I$, the parabolic subgroups $\rho_{I\setminus\{i\}}(B)$ and $\rho_{I\setminus\{i^*\}}(B')$ are opposed.
\end{theorem}

\begin{proof}
    The forward direction follows from the implication \ref{prop:opposedparabolics_borels} $\Rightarrow$ \ref{prop:opposedparabolics_opposed} of \cref{prop:opposedparabolics}. For the backward direction, fix a maximal torus $T\subseteq B\cap B'$. If $\rho_{I\setminus\{i\}}(B)$ and $\rho_{I\setminus\{i^*\}}(B')$ are opposed, then the implication \ref{prop:opposeddualparabolics_opposed} $\Rightarrow$ \ref{prop:opposeddualparabolics_opposite_torus} of \cref{prop:opposeddualparabolics} implies that $\oppborel(T,B) \subseteq \rho_{I\setminus\{i^*\}}(B')$. If this holds for all $i\in I$, then $\oppborel(T,B)\subseteq \bigcap_{i\in I} \rho_{I\setminus\{i^*\}}(B') = B'$. Hence $\oppborel(T,B) = B'$, so $B$ and $B'$ are opposed.
\end{proof}

\section{Opposition is combinatorial}\label{sec:opposition_combinatorial}

\noindent The goal is of this section is to prove that opposition between totally nonnegative Borel subgroups and totally nonpositive Borel subgroups is combinatorial:
\begin{theorem}\label{opposition_cells}
Let $B\in \cB_{\ge 0}$ and $B'\in \cB_{\le 0}$. Then whether $B$ and $B'$ are opposed depends only on the pair of Richardson cells containing $B$ and $B'$.
\end{theorem}

We prove \cref{opposition_cells} at the end of this section. In light of \cref{opposition_cells}, we say that $\Rtp{v}{w}$ and $(\Rtp{v'}{w'})^\perp = \Rtn{w'w_0}{v'w_0}$ are \newword{opposed} if some (equivalently, every) element of $\Rtp{v}{w}$ is opposed to some (equivalently, every) element of $(\Rtp{v'}{w'})^\perp$. In this case, we say that the Bruhat intervals $[v,w]$ and $[v',w']$ are \newword{opposed}; note that this is a symmetric relation since the map $\cdot^\perp$ on $\cB$ preserves opposition. This reduces opposition to a combinatorial relation on the Bruhat intervals of $W$. We study the combinatorics of opposition further in \cref{sec:opposition_combinatorics}. For now, we illustrate combinatorial opposition with an example:
\begin{example}\label{eg:opposition_cells}
Let $G = \SL_3(\CC)$, so that $W = \symgp{3}$. We claim that $[132,231]$ and $[213,312]$ are opposed, i.e., $\Rtp{132}{231}$ and $(\Rtp{213}{312})^\perp = \Rtn{213}{312}$ are opposed. To see this, write
\[
\Rtp{132}{231} = \left\{\begin{bmatrix}
1 & 0 & 0 \\
a & 0 & -1 \\
0 & 1 & 0
\end{bmatrix}\cdot B_+ \;\middle\vert\; a > 0\right\} \quad \text{ and } \quad \Rtp{213}{312} = \left\{\begin{bmatrix}
0 & -1 & 0 \\
1 & 0 & 0 \\
b & 0 & 1
\end{bmatrix}\cdot B_+ \;\middle\vert\; b > 0\right\}.
\]
Then by \cref{opposition_Gaussian}\ref{opposition_Gaussian_plus}, $\Rtp{132}{231}$ and $(\Rtp{213}{312})^\perp$ are opposed if and only if the following matrix lies in $B_-B_+$ (for all $a,b > 0$):
\[
\transposed{\begin{bmatrix}
0 & -1 & 0 \\
1 & 0 & 0 \\
b & 0 & 1
\end{bmatrix}}
\begin{bmatrix}
1 & 0 & 0 \\
a & 0 & -1 \\
0 & 1 & 0
\end{bmatrix}.
\]
This matrix equals
\[
\begin{bmatrix}
a & b & -1 \\
-1 & 0 & 0 \\
0 & 1 & 0
\end{bmatrix}
=
\begin{bmatrix}
1 & 0 & 0 \\[2pt]
-\frac{1}{a} & 1 & 0 \\[2pt]
0 & \frac{a}{b} & 1
\end{bmatrix}
\begin{bmatrix}
a & b & -1 \\[2pt]
0 & \frac{b}{a} & -\frac{1}{a} \\[2pt]
0 & 0 & \frac{1}{b}
\end{bmatrix}
\in B_-B_+,
\]
as desired. (Note that this does not depend on the particular values of $a,b > 0$, in agreement with \cref{opposition_cells}.)
\end{example}

We now turn toward proving \cref{opposition_cells}, which relies on a calculation using a positive weight basis. To this end, for the rest of this section we let $(b_x)_{x\in X}$ denote a positive weight basis for the irreducible representation $V_\lambda$ of $G$, where $\lambda$ is a dominant weight (such a basis exists by \cref{positive_basis_exists}). We denote the unique basis vector of weight $\lambda$ by $b_0$. For each $x\in X$, we define the function 
    \[ \Delta_x: G\to \CC, \quad g \mapsto \langle b_x, g b_0\rangle, \]
where the notation indicates to expand $gb_0$ in the basis $(b_y)_{y\in X}$ and then take the coefficient of $b_x$. Note that since $B_+$ preserves the span of $b_0$, whether $\Delta_x(g)$ vanishes only depends on the Borel subgroup $g\cdot B_+$.

We have the following description of $B_-B_+$:
\begin{lemma}\label{gaussian_coordinate}
Suppose that $\lambda$ is a dominant regular weight, and let $g\in G$. Then $g\in B_-B_+$ if and only if $\Delta_0(g)\neq 0$.
\end{lemma}

\begin{proof}
This follows from the Bruhat decomposition $G = \bigsqcup_{w\in W} B_-\dot{w}B_+$; for details, see the argument in \cite[Proof of 4.3(c)]{lusztig94}.
\end{proof}

The following result explains the choice of signs involved in our choice of Weyl group representatives $\dot{w}$:
\begin{lemma}\label{lem:Worbitbasis}
    For all $w\in W$, the vector $\dot wb_0$ is a positive scalar multiple of $b_x$ for some $x\in X$.
\end{lemma}

\begin{proof}
    We proceed by induction on $\ell(w)$, where the base case $w=e$ holds since $\dot{e}b_0 = b_0$. Now suppose that $\ell(w) > 0$, and write $w = s_iv$ with $\ell(v)<\ell(w)$. By induction, we can assume that $\dot vb_0$ is a positive scalar multiple of some basis vector. By the representation theory of $\SL_2(\CC)$ (and using the identity $\dot{s}_i = \exp(-e_i)\exp(f_i)\exp(-e_i)$) we know that $\dot s_i \dot vb_0 = \frac{f_i^a}{a!} \dot v b_0$, where $a\in\mathbb{N}$ is minimal such that $f_i^{a+1}\dot{v}b_0 = 0$. Since $(b_y)_{y\in X}$ is a positive weight basis, $\frac{f_i^a}{a!} \dot v b_0$ is a nonnegative linear combination of basis vectors. It is also a vector in the $w\lambda$ weight space of $V_{\lambda}$, which has multiplicity one. Hence $\dot wb_0$ is a positive scalar multiple of the unique basis vector $b_x$ with weight $w\lambda$.
\end{proof}

The following is a generalization of a result of Rietsch and Williams \cite[Lemma 6.1]{rietsch_williams08} from Lusztig's canonical basis to any positive weight basis:
\begin{lemma}\label{lem:expansionyi}
    Let $i_1,\ldots,i_k\in I$, $w\in W$, and $x\in X$. Then
    \[ \Delta_x( y_{i_1}(t_1) \cdots y_{i_k}(t_k) \dot w) \] 
    is a polynomial function of $t_1,\ldots, t_k$ with nonnegative coefficients.
\end{lemma}
\begin{proof}
    Let $\theta\in V_\lambda$ be a nonnegative linear combination of basis vectors. We have
    \[y_i(t)\theta = \exp(tf_i)\theta = \theta + tf_i\theta + \frac{t^2}{2!}f_i^2\theta + \cdots, \]
    where all but finitely many terms are $0$. Since $(b_y)_{y\in X}$ is a positive weight basis, the coefficient of $b_x$ in $y_i(t)\theta$ is a polynomial in $t$ with nonnegative coefficients. Repeating this argument shows that $y_{i_1}(t_1)\cdots y_{i_k}(t_k)\theta$ is a polynomial in $t_1,\ldots,t_k$ with nonnegative coefficients. Now we may take $\theta = \dot w b_0$ (which is a nonnegative linear combination of basis vectors by \cref{lem:Worbitbasis}) to finish the proof.
\end{proof}

\begin{lemma}[{Berenstein and Zelevinsky \cite[Theorem 3.1]{berenstein_zelevinsky97}}]\label{lem:braid}
Suppose that $i,j\in I$ satisfy the braid relation $s_is_js_i \cdots = s_js_is_j \cdots$, where each side is a product of $m_{i,j}$ terms. Then we have the relation
\[
y_i(t_1)y_j(t_2)y_i(t_3) \cdots = y_j(t'_1)y_i(t'_2)y_j(t'_3) \cdots \quad \text{ in } G
\]
(with $m_{i,j}$ terms on each side), where every $t'_b$ is a ratio of nonzero polynomials in the $t_a$'s with nonnegative coefficients.\hfill\qed
\end{lemma}

\begin{proposition}[{Rietsch and Williams \cite{rietsch_williams08}}]\label{lem:vanishingrichardson}
Let $x\in X$ and $g\cdot B_+, g'\cdot B_+\in\Rtp{v}{w}$. Then $\Delta_x(g) = 0$ if and only if $\Delta_x(g') = 0$. In other words, whether $\Delta_x$ vanishes at $B\in\cB_{\ge 0}$ depends only on the Richardson cell containing $B$.
\end{proposition}
\begin{proof}
        Rietsch and Williams \cite[Proposition 6.2]{rietsch_williams08} proved this result when $G$ is simply-laced and $X$ is the canonical basis. The same proof works for any $G$ and any positive weight basis, using \cref{lem:Worbitbasis,lem:expansionyi,lem:braid}.
\end{proof}

We will need to know that all the coordinates $\Delta_x$ in \cref{lem:vanishingrichardson} are nonzero on the top Richardson cell $\Rcell{e}{w_0}^{>0} = \cB_{>0}$. This was essentially proved by Lusztig in \cite[Section 3]{lusztig94} when $G$ is simply laced and $X$ is the canonical basis. In order to apply Lusztig's argument to non-simply-laced $G$, we need the following technical result:
\begin{lemma}\label{long_parametrization}
Let $g\in U_-^{>0}$ and $i_1, \dots, i_k\in I$. Then there exist $l \ge 0$ and $j_1, \dots, j_l\in I$ such that
\[
g = y_{j_1}(t'_1) \cdots y_{j_l}(t'_l)y_{i_1}(t_1) \cdots y_{i_k}(t_k)
\]
for some $t'_1, \dots, t'_l, t_1, \dots, t_k > 0$.
\end{lemma}

\begin{proof}
Let $\ast$ denote the \newword{Demazure product} on $W$ (also known as the \emph{$0$-Hecke product}), where $v\ast w$ is the maximal element (in Bruhat order) of the form $v'w'$ for $v'\le v$ and $w'\le w$. (The Demazure product is well-defined and associative; see, e.g., \cite[Appendix A]{he_lu11}.) Set $w \coloneqq s_{i_1} \ast \cdots \ast s_{i_k}$. Let $s_{r_1}\cdots s_{r_m}$ be a reduced expression for $w$, and extend this to a reduced expression $s_{j_1}\cdots s_{j_l}s_{r_1}\cdots s_{r_m}$ for $w_0$. Then by the definition of $U_-^{>0}$, we can write
\[
g = y_{j_1}(t'_1) \cdots y_{j_l}(t'_l)y_{r_1}(t''_1) \cdots y_{r_m}(t''_m)
\]
for some $t'_1, \dots, t'_l, t''_1, \dots, t''_m > 0$. It remains to show that
\[
y_{r_1}(t''_1) \cdots y_{r_m}(t''_m) = y_{i_1}(t_1) \cdots y_{i_k}(t_k) \quad \text{ for some } t_1, \dots, t_k > 0.
\]

To prove this, note that we can obtain the word $r_1 \cdots r_m$ from $i_1\cdots i_k$ by performing a sequence of reductions (i.e.\ replacing $ii$ with $i$, for some $i$) and braid moves in some order. By reversing this process, we obtain the word $i_1\cdots i_k$ from $r_1 \cdots r_m$ by performing a sequence of reverse reductions (i.e.\ replacing $i$ with $ii$, for some $i$) and braid moves. We claim that we can perform this same sequence of operations to the product $y_{r_1}(t''_1) \cdots y_{r_m}(t''_m)$ so that we obtain an expression $y_{i_1}(t_1) \cdots y_{i_k}(t_k)$ with $t_1, \dots, t_k > 0$, which completes the proof. To prove the claim, since the starting parameters $t''_1, \dots, t''_m$ are positive, it suffices to check that reverse reductions and braid moves applied to $y_i$'s evaluated at positive parameters can be performed so as to preserve positivity of the parameters. For a reverse reduction $i \mapsto ii$, we can replace $y_i(t)$ (where $t > 0$) with $y_i(t')y_i(t'')$, where $t'$ and $t''$ are arbitrary positive parameters such that $t' + t'' = t$. For braid moves, positivity of the parameters is preserved by \cref{lem:braid}.
\end{proof}

\begin{proposition}\label{prop:positivitynegativity}
Let $g\in G$.
\begin{enumerate}[label=(\roman*), leftmargin=*, itemsep=2pt]
\item\label{prop:positivitynegativity_strict} If $g\cdot B_+\in\cB_{>0}$, there is a nonzero $c\in\CC$ such that $c\Delta_x(g)>0$ for all $x\in X$.
\item\label{prop:positivitynegativity_weak} If $g\cdot B_+\in\cB_{\ge 0}$, there is a nonzero $c\in\CC$ such that $c\Delta_x(g)\ge 0$ for all $x\in X$.
\end{enumerate}
\end{proposition}

\begin{proof}
It suffices to prove \ref{prop:positivitynegativity_strict}; then \ref{prop:positivitynegativity_weak} follows by taking the closure. By the definition of $\cB_{>0}$, we can write $g = uh$ for some $h\in B_+$ and $u\in U_-^{>0}$. Then $gb_0 = cub_0$ for some nonzero $c\in\CC$, so it suffices to show that $\Delta_x(u) > 0$ for all $x\in X$.

Fix $x\in X$. Since $V_\lambda$ is an irreducible representation of $U(\mathfrak{n}_-)$, there exist $i_1, \dots, i_k\in I$ such that $\Delta_x(f_{i_1}\cdots f_{i_k})\neq 0$. By \cref{long_parametrization}, we can write
\[
u = y_{j_1}(t'_1) \cdots y_{j_l}(t'_l)y_{i_1}(t_1) \cdots y_{i_k}(t_k)
\]
for some $t'_1, \dots, t'_l, t_1, \dots, t_k > 0$. Then writing $y_i(t) = \sum_{a \ge 0}\frac{t^a}{a!}f_i^a$ gives
\[
\Delta_x(u) = \sum_{a'_1, \dots, a'_l, a_1, \dots, a_k \ge 0}\frac{(t'_1)^{a'_1} \cdots (t'_l)^{a'_l}t_1^{a_1} \cdots t_k^{a_k}}{a'_1! \cdots a'_l! a_1! \cdots a_k!}\Delta_x(f_{j_1}^{a'_1} \cdots f_{j_l}^{a'_l}f_{i_1}^{a_1}\cdots f_{i_k}^{a_k}).
\]
Now $\Delta_x(f_{j_1}^{a'_1} \cdots f_{j_l}^{a'_l}f_{i_1}^{a_1}\cdots f_{i_k}^{a_k})$ is always nonnegative since $(b_y)_{y\in X}$ is a positive weight basis, and it is nonzero when $a'_1 = \cdots = a'_l = 0$ and $a_1 = \cdots = a_k = 1$ by assumption. Thus $\Delta_x(u) > 0$.
\end{proof}

Finally, we will need the dual notion to a positive weight basis. We say that a basis $(b'_{x'})_{x'\in X'}$ of $V_\lambda$ is a \newword{negative weight basis} if it is a weight basis such that for all $i\in I$ and $y\in X'$, the vector $-f_ib_y$ has nonnegative coefficients when expanded in the basis $(b'_{x'})_{x'\in X'}$. Equivalently, a negative weight basis is a positive weight basis with respect to the pinning of $G$ obtained by replacing $x_i$ and $y_i$ by their inverses for all $i\in I$ (i.e.\ negating $e_i$ and $f_i$).

Associated to the positive weight basis $(b_x)_{x\in X}$ of $V_\lambda$ is the dual basis $(b_x^*)_{x\in X}$ of the dual representation $V_\lambda^* \cong V_{-w_0\lambda}$. We denote the pairing $V_{\lambda}^*\otimes V_\lambda \to \CC$ by $(\cdot,\cdot)$.
\begin{lemma}\label{dual_negative}
    The dual basis $(b_x^*)_{x\in X}$ is a negative weight basis of $V_{-w_0\lambda}$.
\end{lemma}

\begin{proof}
For every $i\in I$ and $x\in X$, the vector $f_i b_x^*$ is, by definition, the unique vector such that $(f_ib_x^*, b_y) = -(b_x^*,f_ib_y)$ for all $y\in Y$. Note that $(b_x^*,f_ib_y)= \langle b_x, f_ib_y\rangle \geq 0$. Hence $-f_ib_x^*$ is a nonnegative linear combination of the dual basis vectors $(b_y^*)_{y\in Y}$.
\end{proof}

Note that $b_0^*$ is the \emph{lowest} weight vector of $V_{-w_0\lambda}$, so the analog of \cref{lem:Worbitbasis} for negative weight bases implies that $(\dot{w}_0)^{-1} b_0^*$ is a positive multiple of the highest weight vector in $(b_x^*)_{x\in X}$. Let $\Delta^*_x(g)$ denote the coefficient of $b_x^*$ in the basis expansion of $g(\dot{w}_0)^{-1} b_0^*$. Note that whether $\Delta^*_x(g)$ vanishes only depends on the Borel subgroup $g\cdot B_+$. Also, we have
\begin{align}\label{pairing_calculation}
\Delta^*_x(g) = (g(\dot{w}_0)^{-1}b_0^*, b_x) = (b_0^*, \dot w_0 g^{-1}b_x) = \langle b_0, \dot{w}_0g^{-1}b_x\rangle,
\end{align}
i.e., $\Delta^*_x(g)$ equals the coefficient of $b_0$ in the basis expansion of $\dot{w}_0g^{-1}b_x$.

We have the following analogs of \cref{lem:vanishingrichardson,prop:positivitynegativity}:
\begin{proposition}\label{cor:vanishingrichardson}
Let $x\in X$ and $g\cdot B_+, g'\cdot B_+\in\Rtn{v}{w}$. Then $\Delta_x^*(g) = 0$ if and only if $\Delta_x^*(g') = 0$. In other words, whether $\Delta_x^*$ vanishes at $B\in\cB_{\le 0}$ depends only on the Richardson cell containing $B$.
\end{proposition}

\begin{proof}
This follows from \cref{lem:vanishingrichardson,dual_negative}.
\end{proof}

\begin{proposition}\label{prop:positivitynegativity_dual}
Let $g\in G$.
\begin{enumerate}[label=(\roman*), leftmargin=*, itemsep=2pt]
\item\label{prop:positivitynegativity_dual_strict} If $g\cdot B_+\in\cB_{<0}$, there is a nonzero $c\in\CC$ such that $c\Delta_x^*(g)>0$ for all $x\in X$.
\item\label{prop:positivitynegativity_dual_weak} If $g\cdot B_+\in\cB_{\le 0}$, there is a nonzero $c\in\CC$ such that $c\Delta_x^*(g)\ge 0$ for all $x\in X$.
\end{enumerate}
\end{proposition}

\begin{proof}
This follows from \cref{prop:positivitynegativity} by the definitions and by \cref{dual_negative}.
\end{proof}

We are now ready to prove \cref{opposition_cells}:
\begin{proof}[Proof of \cref{opposition_cells}]
Write $B = g\cdot B_+$ and $B' = h\cdot B_+ = h(\dot{w}_0)^{-1}\cdot B_-$ for some $g,h\in G$. By \cref{opposition_Gaussian}\ref{opposition_Gaussian_minus}, we have that $B$ and $B'$ are opposed if and only if $\dot{w}_0h^{-1} g\in B_- B_+$. By \cref{gaussian_coordinate} this is in turn equivalent to $\Delta_0(\dot w_0h^{-1}g) \neq 0$, where we take $\lambda$ to be a regular weight. Writing $gb_0 = \sum_{x\in X}\Delta_x(g)b_x$, we have
\begin{align}\label{coordinates_opposed}
\Delta_0(\dot w_0h^{-1}g) = \sum_{x\in X}\Delta_x(g)\langle b_0, \dot{w}_0h^{-1}b_x\rangle = \sum_{x\in X} \Delta_x(g) \Delta_x^*(h),
\end{align}
where the last equality follows by \eqref{pairing_calculation}. By \cref{prop:positivitynegativity,prop:positivitynegativity_dual}, we can rescale $g$ and $h$ so that the right-hand side above is a sum of nonnegative numbers. Hence the sum is zero if and only if there exists an $x\in X$ such that $\Delta_x(g)$ and $\Delta_x^*(h)$ are both nonzero. By \cref{lem:vanishingrichardson,cor:vanishingrichardson}, whether such an $x$ exists depends only on the pair of Richardson cells containing $B$ and $B'$.
\end{proof}

We see from this proof that determining whether two Bruhat intervals are opposed reduces to finding which coordinates $\Delta_x$ are nonvanishing on a given Richardson cell $\Rtp{v}{w}$. This raises the following problem, which potentially depends on the choice of basis $(b_x)_{x\in X}$:
\begin{problem}\label{problem_coordinates}
Given a Bruhat interval $[v,w]$ of $W$, characterize which functions $\Delta_x$ (for $x\in X$) are nonvanishing on $\Rtp{v}{w}$.
\end{problem}

\section{Combinatorics of opposition on Bruhat intervals}\label{sec:opposition_combinatorics}

\noindent In this section we prove various results on opposition between two Bruhat intervals of $W$. In particular, we show that if two Bruhat intervals intersect, then they are opposed (\cref{intervals_intersect}); we provide a complete characterization in type $A$ (\cref{opposition_A}); and we show that if two Bruhat intervals are opposed, then their Bruhat interval polytopes intersect (\cref{BIPs}).

\subsection{Intersection implies opposition}

\begin{theorem}\label{intervals_intersect}
If two Bruhat intervals of $W$ intersect, then they are opposed.
\end{theorem}

\begin{proof}
Proceed by contradiction and suppose that there exist non-opposed Bruhat intervals $[v,w]$ and $[v',w']$ which also intersect. Then by \cref{opposition_Gaussian}\ref{opposition_Gaussian_plus}, we have
\begin{align}\label{intervals_intersect_limit}
\transpose{h}g \notin B_-B_+ \quad \text{ for all } g\cdot B_+\in \Rtp{v}{w} \text{ and } h\cdot B_+\in\Rtp{v'}{w'}.
\end{align}
Now let $x\in [v,w]\cap [v',w']$. Since $\Rtp{x}{x} = \{\dot{x}\cdot B_+\}$, by \eqref{richardson_closure_tnn} we have that $\dot{x}\cdot B_+$ is in the Euclidean closure of both $\Rtp{v}{w}$ and $\Rtp{v'}{w'}$. Then since $B_-B_+$ is open, \eqref{intervals_intersect_limit} implies that $\transpose{\dot{x}}\dot{x}\notin B_-B_+$. But $\transpose{\dot{x}}\dot{x} = \dot{e}$ which is in $B_-B_+$, a contradiction.
\end{proof}

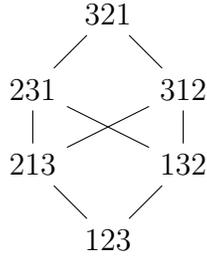
\begin{figure}
    \centering
    \begin{tikzpicture}
        \node (B) at (0,0) {$123$};
        \node (L1) at (-1,1) {$213$};
        \node (L2) at (-1,2) {$231$};
        \node (T) at (0,3) {$321$};
        \node (R2) at (1,2) {$312$};
        \node (R1) at (1,1) {$132$};
        \draw (B) -- (L1) -- (L2) -- (T) -- (R2) -- (R1) -- (B);
        \draw (L1) -- (R2);
        \draw (R1) -- (L2);
    \end{tikzpicture}
    \caption{Bruhat order on $\symgp{3}$.}
    \label{fig:S3}
\end{figure}

\begin{example}\label{eg:intervals_intersect}
Let $G = \SL_3(\CC)$, so that $W = \symgp{3}$ (depicted in \cref{fig:S3}). Recall from \cref{eg:opposition_cells} that the intervals $[132,231]$ and $[213,312]$ are opposed; this gives an example of opposed Bruhat intervals which do not intersect (see \cref{eg:opposition_S4} for another example). We can check that this pair is the only unordered pair of opposed Bruhat intervals of $\symgp{3}$ which do not intersect.
\end{example}

As a special case of \cref{intervals_intersect}, we deduce that every totally nonnegative element of $\cB$ is opposed to every totally negative element; this was proved for $G = \SL_n(\CC)$ by Blayac, Hamenst\"{a}dt, Marty, and Monti \cite[Lemma 5.2]{blayac_hamenstadt_marty_monti}. In particular, every totally positive element is opposed to every totally negative element, which recovers a result (for arbitrary $G$) of Lusztig \cite[Proposition 1.2]{lusztig}.
\begin{corollary}\label{tnn_tp_opposed}
Every element of $\cB_{\ge 0}$ is opposed to every element of $\cB_{<0}$. Similarly, every element of $\cB_{>0}$ is opposed to every element of $\cB_{\le 0}$.
\end{corollary}

\begin{proof}
We have $\cB_{<0} = (\Rtp{e}{w_0})^\perp$, and the Bruhat interval $[e,w_0] = W$ intersects every Bruhat interval of $W$. Hence the first assertion follows from \cref{intervals_intersect}. The second assertion follows similarly.
\end{proof}

We have the following intuitive consequence of \cref{tnn_tp_opposed}:
\begin{corollary}\label{borel_torus_containment}
Every element of $\cB_{\ge 0}$ contains an element of $\cT_{\ge 0}$.
\end{corollary}

\begin{proof}
Given $B\in\cB_{\ge 0}$, take any $B'\in\cB_{<0}$. Then $B\cap B'\in\cT_{\ge 0}$ by \cref{tnn_tp_opposed,closure_equality}.
\end{proof}

\cref{tnn_tp_opposed} also allows us to prove an extension of 
\cref{tp_torus_parametrization}:
\begin{lemma}\label{opposed_torus_parametrization}
Let $S\subseteq\cB_{\ge 0}$ and $S'\subseteq\cB_{\le 0}$ such that every element of $S$ is opposed to every element of $S'$. Then the following map is injective:
\[
S\times S' \to \cT_{\ge 0}, \quad (B,B') \mapsto B\cap B'.
\]
\end{lemma}

\begin{proof}
The proof is essentially the same as in \cite[Proposition 1.3]{lusztig}, which we repeat here for completeness. Given $(B_1,B_1'), (B_2,B_2')\in S\times S'$ with $B_1\cap B_1' = B_2\cap B_2' = T$, we must show that $(B_1,B_1') = (B_2,B_2')$. By assumption, $B_1'$ is opposed to both $B_1$ and $B_2$, so $B_1 = \oppborel(T,B'_1) = B_2$. Similarly $B_1' = B_2'$. 
\end{proof}

\begin{corollary}\label{tnn_tp_torus_parametrization}~
\begin{enumerate}[label=(\roman*), leftmargin=*, itemsep=2pt]
\item\label{tnn_tp_torus_parametrization1} The map $\cB_{\ge 0} \times \cB_{<0} \to \cT_{\ge 0}$, $(B,B') \mapsto B\cap B'$ is injective.
\item\label{tnn_tp_torus_parametrization2} The map $\cB_{>0} \times \cB_{\le 0} \to \cT_{\ge 0}$, $(B,B') \mapsto B\cap B'$ is injective.
\end{enumerate}
\end{corollary}

\begin{proof}
This follows from \cref{opposed_torus_parametrization,tnn_tp_opposed}.
\end{proof}

\subsection{Opposition in type \texorpdfstring{$A$}{A}}\label{sec:opposition_A}
We consider opposition in the type $A$ case (when $G = \SL_n(\CC)$, adopting the conventions of \cref{eg:pinning,eg:flag_variety}.

\begin{proposition}\label{opposition_complete_flags}
Let $G = \SL_n(\CC)$, and let $F_\bullet$ and $F'_\bullet$ be complete flags in $\Fl_n$. Then $F_\bullet$ and $F'_\bullet$ are opposed (in the sense of the corresponding Borel subgroups being opposed) if and only if the subspaces $F_k$ and $F'_{n-k}$ are transverse for all $k\in [n-1]$.
\end{proposition}

\begin{proof}
This is just a restatement of \cref{opposition_maximal_parabolics} in the case $G = \SL_n(\CC)$, using the fact (proved in \cref{prop:transversality_opposition}) that opposition between maximal parabolic subgroups of $\SL_n(\CC)$ corresponds to transversality of subspaces. (Alternatively, we can prove this directly using \cref{opposition_Gaussian} and the fact that $B_-B_+$ consists of matrices whose leading principal minors are nonzero.)
\end{proof}

Recall that the positroid of $V\in\Gr_{k,n}^{\ge 0}$ is $\{I\in\binom{[n]}{k} \mid \Delta_I(V)\neq 0\}$. We need the following result:
\begin{lemma}[{Tsukerman and Williams \cite[Theorem 7.1]{tsukerman_williams15}}]\label{positroid_projection}
Let $F_\bullet\in\Fl_n^{\ge 0}$ be such that the corresponding Borel subgroup lies in the Richardson cell $\Rtp{v}{w}$. Then for all $k\in [n-1]$, the positroid of $F_k\in\Gr_{k,n}^{\ge 0}$ is $\{x([k]) \mid x\in [v,w]\}$.\hfill\qed
\end{lemma}

We now state our characterization of opposition in type $A$:
\begin{theorem}\label{opposition_A}
Let $G = \SL_n(\CC)$, so that $W = \symgp{n}$. Then the Bruhat intervals $[v,w]$ and $[v',w']$ of $W$ are opposed if and only if for all $k\in [n-1]$, there exist $x\in [v,w]$ and $x'\in [v,w]$ such that $x([k]) = x'([k])$.
\end{theorem}

\begin{proof}
Let $F_\bullet$ and $F'_\bullet$ be complete flags whose corresponding Borel subgroups lie in $\Rtp{v}{w}$ and $\Rtp{v'}{w'}$, respectively. We have the following chain of equivalent statements:
\begin{align*}
&\mathrel{\hphantom{\Leftrightarrow}\,} [v,w] \text{ and } [v',w'] \text{ are opposed} \\
&\Leftrightarrow\, F_\bullet \text{ and } (F'_\bullet)^\perp \text{ are opposed} \quad \text{(by \cref{opposition_cells})} \\
&\Leftrightarrow\, F_k \text{ and } (F'_k)^\perp \text{ are transverse for all } k\in [n-1] \quad \text{(by \cref{opposition_complete_flags} and \eqref{perp_complete_flag})} \\
&\Leftrightarrow\, \text{the positroids of $F_k$ and $F'_k$ intersect for all $k\in [n-1]$} \quad \text{(by \cref{positroid_opposition})} \\
&\Leftrightarrow\, \{x([k]) \mid x\in [v,w]\}\cap\{x'([k]) \mid x'\in [v',w']\}\neq\varnothing \text{ for all } k\in [n-1]
\end{align*}
by \cref{positroid_projection}, as desired.
\end{proof}

\begin{example}\label{eg:opposition_S3}
Recall from \cref{eg:opposition_cells} that the following Bruhat intervals of $W = \symgp{3}$ are opposed:
\[
[132,231] = \{132,231\}
\quad \text{ and } \quad
[213,312] = \{213,312\}.
\]
Let us instead use \cref{opposition_A} to verify that these intervals are opposed. For $k=1$, we can (in fact, must) take $x = 231$ and $x' = 213$, whence $x([k]) = x'([k]) = \{2\}$. For $k=2$, we can (in fact, must) take $x = 132$ and $x' = 312$, whence $x([k]) = x'([k]) = \{1,3\}$. That these two intervals do not intersect is reflected in the fact that $\{2\}$ is not contained in $\{1,3\}$.
\end{example}

\begin{example}\label{eg:opposition_S4}
Let us use \cref{opposition_A} to verify that the following Bruhat intervals of $W = \symgp{4}$ are opposed:
\[
[1342,2341] = \{1342,2341\}
\quad \text{ and } \quad
[2314,3412] = \{2314, 2413, 3214, 3412\}.\vspace*{6pt}
\]
\begin{center}
\setlength{\tabcolsep}{12pt}
\setlength{\extrarowheight}{3pt}
\;\begin{tabular}{|c|c|c|}\hline
$k$ & $x$ & $x'$ \\ \hline
$1$ & $2341$ & $2314$ or $2413$ \\ \hline
$2$ & $2341$ & $2314$ \\ \hline
$3$ & $1342$ & $3412$ \\ \hline
\end{tabular}\;
\end{center}\vspace*{6pt}
This is another example of opposed Bruhat intervals which do not intersect.
\end{example}

\begin{problem}
Find a combinatorial characterization of opposition for Bruhat intervals (such as the one in \cref{opposition_A} for type $A$).
\end{problem}

It would be particularly interesting to have a type-free description of opposition which makes sense for arbitrary Coxeter groups $W$ (including non-crystallographic ones).

\subsection{Opposition implies intersecting Bruhat interval polytopes}\label{sec:BIPs}

Recall that $(\varpi_i)_{i\in I}$ denote the fundamental weights of $G$. Set $\rho \coloneqq \sum_{i\in I}\varpi_i$. The \newword{Bruhat interval polytope} of $[v,w]$ is defined to be the convex hull of
\[
\{x\rho \mid x\in [v,w]\}.
\]
Bruhat interval polytopes were introduced by Kodama and Williams \cite{kodama_williams15} when $W = \symgp{n}$, and by Tsukerman and Williams \cite{tsukerman_williams15} for general Coxeter groups $W$. When $W = \symgp{n}$ (i.e.\ $G = \SL_n(\CC)$), we can identify the Bruhat interval polytope with the polytope in $\RR^n$ whose vertices are
\begin{align}\label{BIP_explicit}
\{(x^{-1}(1), \dots, x^{-1}(n)) \mid x\in [v,w]\}.
\end{align}

We show that the geometry of Bruhat interval polytopes provide a necessary condition for two Bruhat intervals to be opposed:
\begin{theorem}\label{BIPs}
    Let $[v,w]$ and $[v',w']$ be opposed Bruhat intervals of $W$. Then the Bruhat interval polytopes of $[v,w]$ and $[v',w']$ intersect.
\end{theorem}
\begin{proof}
    By definition, there exist opposed Borel subgroups $g\cdot B_+ \in \Rtp{v}{w}$ and $h\cdot B_+ \in (\Rtp{v'}{w'})^\perp = \Rtn{w'w_0}{v'w_0}$. Let $(b_x)_{x\in X}$ be a positive weight basis of $V_\rho$. Then as in the proof of \cref{opposition_cells} (see \eqref{coordinates_opposed}), since $\rho$ is a dominant regular weight, there exists $x\in X$ so that $\Delta_x(g)$ and $\Delta_x^*(h)$ are both nonzero. By \cite[Proposition 6.20 and Theorem 7.1]{tsukerman_williams15}, the Bruhat interval polytope of $[v,w]$ is the convex hull of $\{\wt(b_y) \mid \Delta_y(g)\neq 0\}$. Dually, by considering the pinning of $G$ obtained by inverting every $x_i$ and $y_i$, we get that the Bruhat interval polytope of $[w'w_0, v'w_0]$ is the convex hull of $\{\wt(b_y^*)\mid \Delta^*_y(h)\neq 0\}$. Since $\wt(b_y^*) = -\!\wt(b_y)$ and $w_0\rho = -\rho$, we get that the Bruhat interval polytope of $[v',w']$ is the convex hull of $\{\wt(b_y)\mid \Delta^*_y(h)\neq 0\}$. Therefore the Bruhat interval polytopes of $[v,w]$ and $[v',w']$ intersect at $\wt(b_x)$.
\end{proof}

\begin{example}\label{intersecting_BIPs}
We show that the converse of \cref{BIPs} does not hold in general, i.e., there exist non-opposed Bruhat intervals whose Bruhat interval polytopes intersect. To see this, consider the following Bruhat intervals in $W = \symgp{4}$:
\[
[1342,2341] = \{1342,2341\}
\quad \text{ and } \quad
[3124,4123] = \{3124, 4123\}.
\]
By taking $k=1$ in \cref{opposition_A}, we see that the intervals are not opposed. Also, using \eqref{BIP_explicit} we see that their Bruhat interval polytopes intersect in the point $(2,3,2,3)$.
\end{example}

\begin{corollary}\label{singleton}
    The Bruhat intervals $[v,w]$ and $[x,x]$ of $W$ are opposed if and only if $x\in [v,w]$.
\end{corollary}

\begin{proof}
The backward direction follows from \cref{intervals_intersect}. For the forward direction, suppose that $[v,w]$ and $[x,x]$ are opposed. Then by \cref{BIPs}, the Bruhat interval polytope of $[v,w]$ contains $x\rho$, so $x\in [v,w]$.
\end{proof}

\section{Proof of Lusztig's conjecture}\label{sec:lusztig_proof}

\noindent In this section we prove the conjecture of Lusztig from \cite[Section 5]{lusztig}. We begin by recalling the conjecture and its consequence for the space $\cT_{>0}$ of totally positive maximal tori.

\subsection{Statement of Lusztig's conjecture}
Recall that $T_0$ denotes the standard torus of $G$. For $p > 0$ we define
\[
(T_0)_{>0}^{(p)} \coloneqq \{t\in (T_0)_{>0}\mid \alpha_i(t)>p \text{ for all } i\in I\}.
\]
Then Lusztig's conjecture is the following:
\begin{theorem}\label{lusztig_conjecture}
Let $T\in\cT_{>0}$. Write $T = g\cdot T_0$ for some $g\in G$ such that $g\cdot B_+\in\cB_{>0}$, $g\cdot B_-\in\cB_{<0}$, and $g = g_1g_2$ with $g_1\in U_-^{>0}$ and $g_2\in U_+^{<0}$. (This is always possible by \cite[Proof of Proposition 1.2]{lusztig}.)
\begin{enumerate}[label=(\roman*), leftmargin=*, itemsep=2pt]
\item\label{lusztig_conjecture1} We have $T\cap G_{>0} \subseteq g\cdot (T_0)_{>0}^{(1)}$.
\item\label{lusztig_conjecture2} We have $g\cdot (T_0)_{>0}^{(p)} \subseteq T\cap G_{>0}$ for some $p > 0$.
\end{enumerate}

\end{theorem}

We will prove \cref{lusztig_conjecture} in \cref{conjecture_proof}. We point out that since $(T_0)_{>0}^{(p)} \supseteq (T_0)_{>0}^{(q)}$ for all $0 < p < q$, if \cref{lusztig_conjecture}\ref{lusztig_conjecture2} holds for some value of $p > 0$, then it holds for all values $\ge p$.

We now explain the motivation behind Lusztig's conjecture, following \cite[Sections 2.3 and 5]{lusztig}. Given $h\in G_{>0}$, by \cite[Theorem 8.9(a)]{lusztig94} there exists a unique $B\in\cB_{>0}$ containing $h$ and a unique $B'\in\cB_{<0}$ containing $h$. Let $\pi': G_{>0} \to \cT_{>0}$ denote the map sending $h$ to $B\cap B'$. Equivalently, since $h$ is totally positive, it is contained in a unique maximal torus $T$ \cite[Theorem 5.6]{lusztig94}; we have $\pi'(h) = T$.

\begin{corollary}\label{conjecture_consequence}
Every $T\in\cT_{>0}$ contains an element of $G_{>0}$. That is, the map $\pi': G_{>0} \to \cT_{>0}$ is surjective.
\end{corollary}

\begin{proof}
Let $T\in\cT_{>0}$. Then \cref{lusztig_conjecture}\ref{lusztig_conjecture2} implies that $T$ contains some $h\in G_{>0}$. Since $h\in T$ we have $\pi'(h) = T$, so $\pi'$ is surjective.
\end{proof}

We show in \cref{conjecture_consequence_counterexample} that a natural extension of \cref{conjecture_consequence} to $\cT_{\ge 0}$ fails to hold. We now illustrate \cref{lusztig_conjecture} with an example, which also provides some intuition about the proof to follow:
\begin{example}\label{conjecture_example}
Let $G = \SL_3(\CC)$, and adopt the setup of \cref{eg:lusztig_intro}. That is,
\[
g = 
\begin{bmatrix}
1 & -0.5 & 0.4 \\
1 & 0 & -0.2 \\
1 & 1 & 0.6
\end{bmatrix}
\quad\text{ and }\quad
T = \left\{g\begin{bmatrix}\lambda_1 & 0 & 0 \\ 0 & \lambda_2 & 0 \\ 0 & 0 & \lambda_3\end{bmatrix}g^{-1}\right\} \subseteq G.
\]
Let us verify \cref{lusztig_conjecture}\ref{lusztig_conjecture2} for this choice of $T$. To this end, set
\[
h = g\begin{bmatrix}\lambda_1 & 0 & 0 \\ 0 & \lambda_2 & 0 \\ 0 & 0 & \lambda_3\end{bmatrix}g^{-1} = \frac{1}{10}\scalebox{0.92}{$\begin{bmatrix}
2\lambda_1 + 4\lambda_2 + 4\lambda_3 & 7\lambda_1 - \lambda_2 - 6\lambda_3 & \lambda_1 - 3\lambda_2 + 2\lambda_3 \\[2pt]
2\lambda_1 - 2\lambda_3 & 7\lambda_1 + 3\lambda_3 & \lambda_1 - \lambda_3 \\[2pt]
2\lambda_1 - 8\lambda_2 + 6\lambda_3 & 7\lambda_1 + 2\lambda_2 - 9\lambda_3 & \lambda_1 + 6\lambda_2 + 3\lambda_3
\end{bmatrix}$}.
\]
Then $h\in g\cdot (T_0)_{>0}^{(p)}$ if and only if $\frac{\lambda_1}{\lambda_2} > p$ and $\frac{\lambda_2}{\lambda_3} > p$. Therefore it suffices to check that $h$ is totally positive (i.e.\ its minors are all positive) for all $\lambda_1 \gg \lambda_2 \gg \lambda_3$.

To see this, we determine the leading term of every minor and verify that it is positive. Indeed, when $\lambda_1 \gg \lambda_2 \gg \lambda_3$, every entry of $h$ behaves like $\frac{1}{5}\lambda_1$, $\frac{7}{10}\lambda_1$, or $\frac{1}{10}\lambda_1$, which are all positive. Also, the top-left $2\times 2$ minor of $h$ behaves like $\frac{3}{10}\lambda_1\lambda_2$, which is positive. We can similarly check that the other eight $2\times 2$ minors behave like some positive constant times $\lambda_1\lambda_2$, and hence are all positive. Finally, $\det(h) = 1$. Therefore $h$ is totally positive.
\end{example}

\subsection{Proof of Lusztig's conjecture}\label{conjecture_proof}
We prove each of the two parts of \cref{lusztig_conjecture}.
\begin{proof}[Proof of \cref{lusztig_conjecture}\ref{lusztig_conjecture1}.]
Let $h\in T\cap G_{>0}$, so that $h = gtg^{-1}$ for some $t\in T_0$. We must show that $t\in (T_0)_{>0}^{(1)}$. By \cite[Corollary 8.10]{lusztig94}, since $h\in G_{>0}$ we can write $h = u_1u_2t'u_1^{-1}$ for some $u_1\in U_-^{>0}$, $u_2\in U_+$, and $t'\in (T_0)_{>0}^{(1)}$. We will show that $t'=t$, which implies $t\in (T_0)_{>0}^{(1)}$, as desired.

Recall that \cite[Theorem 8.9(a)]{lusztig94} implies $h$ is contained in a unique totally positive Borel subgroup, which is necessarily $g \cdot B_+ = g_1\cdot B_+$. Since $u_1\cdot B_+$ is also totally positive and contains $h$, we have $g_1\cdot B_+ = u_1\cdot B_+$. Since $g_1,u_1\in U_-$, we get $g_1 = u_1$. Then from
\[
u_1u_2t'u_1^{-1} = h = gtg^{-1} = g_1g_2tg_2^{-1}g_1^{-1}
\]
we find $t' = u_2^{-1}g_2tg_2^{-1} \in U_+ t U_+$, so $t' = t$. This completes the proof.
\end{proof}

\begin{proof}[Proof of \cref{lusztig_conjecture}\ref{lusztig_conjecture2}.]
Let $t\in g\cdot (T_0)_{>0}^{(p)}$. It suffices to show that when $p\gg 0$ we have $t\in G_{>0}$, which by \cref{fominzelevinsky} is equivalent to every generalized minor of $t$ being positive. To this end, let $\eta_1$ and $\eta_2$ be extremal weight vectors in a fundamental representation $V_{\varpi_i}$. We will show that $\Delta(t) \coloneqq \langle \eta_2, t \eta_1\rangle$ is positive when $p \gg 0$, which completes the proof.

Write $t = g t_0 g^{-1}$ for some $t_0\in (T_0)^{(p)}_{>0}$. Let $(b_x)_{x\in X}$ be a positive weight basis of $V_{\varpi_i}$ (which exists by \cref{positive_basis_exists}), and assume that the highest weight vector $\xi_{\varpi_i}$ of $V_{\varpi_i}$ is the basis vector $b_0$. By \cref{lem:Worbitbasis}, we may rescale the basis vectors by positive scalars so that every extremal weight vector $\eta$ is equal to some basis vector $b_{x_\eta}$. For $\theta\in V_{\varpi_i}$, we let $\langle b_x, \theta\rangle$ denote the coefficient of $b_x$ in the basis expansion of $\theta$. Then
\begin{align*} \Delta(t) = \langle \eta_2, gt_0g^{-1}\eta_1\rangle &= \sum_{x,y\in X} \langle \eta_2, gb_y\rangle  \langle b_y, t_0b_x \rangle \langle b_x, g^{-1} \eta_1\rangle \\ 
 &= \sum_{x\in X} \langle \eta_2, g b_x\rangle t_0^{\mathrm{wt}(b_x)} \langle b_x, g^{-1} \eta_1\rangle. 
 \end{align*}
When $p\gg 0$, this sum is dominated by the term containing $t_0^{\varpi_i}$ (assuming its coefficient is nonzero), i.e., the term where $b_x$ is the highest weight vector $b_0$. The coefficient of $t_0^{\varpi_i}$ is
\[
\langle \eta_2, g b_0\rangle \langle b_0, g^{-1}\eta_1\rangle,
\]
and to finish the proof it suffices to show that this coefficient is positive. We show that in fact
\begin{align}\label{lusztig_conjecture_to_prove}
\langle \eta, gb_0\rangle > 0 \text{ and } \langle b_0, g^{-1}\eta\rangle > 0 \text{ for all extremal weight vectors $\eta$}.
\end{align}

To see that $\langle\eta, gb_0\rangle > 0$, take any $u\in U_+^{>0}$. Since $g_1\in U_-^{>0}$ we have $g_1u\in G_{>0}$, so $\langle\eta, g_1ub_0\rangle > 0$ by \cref{fominzelevinsky}. Because $U_+$ fixes $b_0$, we get
\begin{align}\label{fominzelevinsky_trick}
\langle \eta, g b_0\rangle 
=
\langle \eta, g_1g_2b_2\rangle
=
\langle \eta, g_1b_0\rangle
=
\langle \eta, g_1ub_0\rangle > 0.
\end{align}

To see that $\langle b_0, g^{-1}\eta\rangle > 0$, note that since $g\dot w_0\cdot B_+ = g\cdot B_{-}$ is totally negative, by \cref{prop:positivitynegativity_dual}\ref{prop:positivitynegativity_dual_strict} there exists a nonzero $c\in\CC$ such that $c\Delta^*_x(g\dot{w}_0) > 0$ for all $x\in X$. Then by \eqref{pairing_calculation} we have
\[
c\langle b_0, g^{-1}\eta\rangle = c\langle b_0, g^{-1}b_{x_\eta}\rangle = c\Delta^*_{x_\eta}(g\dot{w}_0) > 0.
\]
It remains to show that $c > 0$. Since $c$ does not depend on $\eta$, it suffices to show that $\langle b_0, g^{-1}\eta\rangle > 0$ for some choice of $\eta$. We take $\eta=\dot w_0b_0$ (a scalar multiple of the lowest weight vector) and argue similarly to \eqref{fominzelevinsky_trick}. That is, take any $u'\in U_-^{>0}$. Since $g_2^{-1}\in U_+^{>0}$, we have $g_2^{-1}u'\in G_{>0}$. Because $U_-$ fixes $\eta$, we get
\[
\langle b_0, g^{-1}\eta\rangle
=
\langle b_0, g_2^{-1}g_1^{-1}\eta\rangle
=
\langle b_0, g_2^{-1}\eta\rangle
=
\langle b_0,g_2^{-1}u'\eta\rangle > 0
\]
by \cref{fominzelevinsky}. This proves \eqref{lusztig_conjecture_to_prove}, as desired.
\end{proof}

\begin{remark}
We mention that in the proof of 
\cref{lusztig_conjecture}\ref{lusztig_conjecture2}, we can take $(b_x)_{x\in X}$ to be any weight basis containing the extremal weight vectors (not necessarily a positive weight basis). 
\end{remark}

\section{Topology of the space of totally nonnegative maximal tori}\label{sec:topology}

\noindent In this section we consider the topology of $\cT_{\ge 0}$ (the closure of the space $\cT_{>0}$ of totally positive maximal tori). It is instructive to first consider an example which shows that $\cT_{\ge 0}$ lacks some of the nice topological properties of other totally nonnegative spaces; in particular, $\cT_{\ge 0}$ is neither compact nor contractible in general. We then propose working with the space $\framed_{\ge 0}$ of totally nonnegative \emph{framed} maximal tori instead, which seems to have better topological properties.
\begin{example}\label{eg:square}
Let $G = \SL_2(\CC)$, so that two elements of $\cB$ are opposed if and only if they are distinct. We can identify $\cB$ with $\PP^1 = \CC\cup\{\infty\}$, which identifies $\cB_{\ge 0}$ with $[0,\infty]$ and $\cB_{\le 0}$ with $[-\infty,0]$. This identifies $\cB_{\ge 0}\times\cB_{\le 0}$ with the square below.
\[ \cB_{\geq 0} \times \cB_{\leq 0} =
\begin{tikzpicture}[baseline={([yshift=-.5ex]current bounding box.center)}]
    \draw[fill=gray!20] (0,0) -- (1,0) -- (1,1) -- (0,1) -- cycle;
        \fill (0,1) circle (2pt) (1,0) circle (2pt);
        \fill (0,0) circle (2pt) (1,1) circle (2pt);
    \node[below left] at (0,0) {\footnotesize$(0,-\infty)$};
    \node[below right] at (1,0) {\footnotesize$(\infty,-\infty)$};
    \node[above left] at (0,1) {\footnotesize$(0,0)$};
    \node[above right] at (1,1) {\footnotesize$(\infty,0)$};
    \end{tikzpicture}
\]

The map \eqref{torus_bijection} is a homeomorphism from the interior of the square to $\cT_{>0}$. To obtain $\cT_{\ge 0}$ from the square, we must delete the two corners $(0,0)$ and $(\infty,-\infty)$ from the square, since they correspond to non-opposed Borel subgroups.  We must also identify the other two corners $(0,-\infty)$ and $(\infty,0)$, because they correspond to the same totally nonnegative maximal torus (namely $T_0$), as shown below.
\[
\cT_{\geq 0} =\begin{tikzpicture}[baseline={([yshift=-.5ex]current bounding box.center)}]
    \draw[fill=gray!20] (0,0) -- (1,0) -- (1,1) -- (0,1) -- cycle;
        \draw[dotted,fill=white] (0,1) circle (2pt) (1,0) circle (2pt);
        \fill (0,0) circle (2pt) (1,1) circle (2pt);       
        \draw[blue, <->] (0.1,0.1) -- (.9,.9);
    \node[below left] at (0,0) {\footnotesize$(0,-\infty)$};
    \node[above right] at (1,1) {\footnotesize$(\infty,0)$};
        
    \end{tikzpicture}
\]
In particular, we see that $\cT_{\ge 0}$ is neither compact nor contractible.

In terms of opposed Bruhat intervals, every unordered pair of intervals of $W = \symgp{2}$ is opposed except for $[e,e]$ and $[w_0,w_0]$, corresponding to the two corners deleted above.
\end{example}

If we wish to avoid making identifications such as in \cref{eg:square} (which seems desirable to avoid topological complications), we can work with the space of framed maximal tori $\framed$ from \cref{background_tori}. We define the space of \newword{totally positive framed maximal tori} by
\[
\framed_{>0} \coloneqq \{(T,B)\in\framed \mid T\in\cT_{>0} \text{ and } B\in\cB_{>0}\}.
\]
We define the space $\framed_{\ge 0}$ of \newword{totally nonnegative framed maximal tori} to be the Euclidean closure of $\framed_{>0}$.

Note that there is an apparent asymmetry in the definition, since in a totally positive framed torus we only pick a $B\in\cB_{>0}$ containing $T\in\cT_{>0}$ but not a $B'\in\cB_{<0}$ containing $T$. However, the following result shows that the definition is indeed symmetric:
\begin{lemma}\label{framed_tori_opposite}~
\begin{enumerate}[label=(\roman*), leftmargin=*, itemsep=2pt]
\item\label{framed_tori_opposite_positive} We have $\oppborel(T,B)\in\cB_{<0}$ for all $(T,B)\in\framed_{>0}$.
\item\label{framed_tori_opposite_nonnegative} We have $\oppborel(T,B)\in\cB_{\le 0}$ for all $(T,B)\in\framed_{\ge 0}$.
\end{enumerate}
\end{lemma}

\begin{proof}
It suffices to prove \ref{framed_tori_opposite_positive}; then \ref{framed_tori_opposite_nonnegative} follows by taking the closure. Let $B' = \oppborel(T,B)$. Since $T\in\cT_{>0}$, we can write $T = B_1\cap B_1'$ for some $B_1\in\cB_{>0}$ and $B_1'\in\cB_{<0}$. Since $B_1$ is opposed to both $B'$ (by \cref{tnn_tp_opposed}) and $B_1'$, we have $B' = \oppborel(T,B_1) = B_1'\in\cB_{<0}$.
\end{proof}

We have the following interesting consequences of \cref{framed_tori_opposite}:
\begin{proposition}\label{framed_explicit}
The space of totally nonnegative framed maximal tori is described as
\[
\framed_{\ge 0} = \{(T,B)\in\framed \mid T\in\cT_{\ge 0}, B\in\cB_{\ge 0}, \textnormal{ and } \oppborel(T,B)\in\cB_{\le 0}\}.
\]
\end{proposition}

\begin{proof}
The $\subseteq$ containment follows from the definitions and \cref{framed_tori_opposite}\ref{framed_tori_opposite_nonnegative}. The $\supseteq$ containment follows from \cref{closure_equality}.
\end{proof}

Perhaps surprisingly, the condition $\oppborel(T,B)\in\cB_{\le 0}$ cannot be removed from \cref{framed_explicit}, as we will prove later in \cref{framed_explicit_necessary}.

\begin{corollary}\label{framed_tori_unique}
Let $(T,B)\in\framed_{\ge 0}$ such that $B\in\cB_{>0}$. Then $(T,B)$ is the unique totally nonnegative framing of $T$, i.e., if $(T,\tilde{B})\in\framed_{\ge 0}$ then $\tilde{B} = B$.
\end{corollary}

\begin{proof}
Let $(T,\tilde{B})\in\framed_{\ge 0}$, and let $B' = \oppborel(T,\tilde{B})$. By \cref{framed_tori_opposite}\ref{framed_tori_opposite_nonnegative} we have $B'\in\cB_{\le 0}$, so $B'$ is opposed to $B$ by \cref{tnn_tp_opposed}. Hence $\tilde{B} = \oppborel(T,B') = B$.
\end{proof}

We leave it as an open problem to determine all totally nonnegative Borel subgroups containing a given $T\in\cT_{\ge 0}$, as well as all totally nonnegative framings:
\begin{problem}\label{problem_framings}
Let $T\in\cT_{\ge 0}$.
\begin{enumerate}[label=(\roman*), leftmargin=*, itemsep=2pt]
\item\label{problem_framings1} Find all $B\in\cB_{\ge 0}$ containing $T$.
\item\label{problem_framings2} Find all totally nonnegative framings $(T,B)\in\framed_{\ge 0}$ of $T$.
\end{enumerate}
\end{problem}

Again, we emphasize that while every Borel subgroup $B$ appearing in \ref{problem_framings2} above gives an answer to \ref{problem_framings1}, the converse does not always hold (by \cref{framed_explicit_necessary}). To illustrate \cref{problem_framings}, we consider the example when $T = T_0$ (the standard torus):
\begin{example}\label{eg:framings}
The Borel subgroups containing $T_0$ are precisely $w\cdot B_+$ for $w\in W$. By the calculation in \cref{eg:closure_equality}, we have $w\cdot B_+\in\cB_{\ge 0}$ and the framing $(T_0,w\cdot B_+)$ is totally nonnegative (for all $w\in W$). This answers \cref{problem_framings} when $T = T_0$.
\end{example}

We now consider the topology of $\framed_{>0}$ and $\framed_{\ge 0}$.
\begin{theorem}\label{framed_cell_decomposition}
We have the cell decomposition
\begin{align*}
\framed_{\ge 0} &= \bigsqcup_{([v,w],[v',w'])}\{((B\cap B'), B) \mid B\in\Rtp{v}{w}, B'\in(\Rtp{v'}{w'})^\perp\} \\
&\cong \bigsqcup_{([v,w],[v',w'])}\RR_{>0}^{\ell(w) + \ell(w') - \ell(v) - \ell(v')},
\end{align*}
where the disjoint unions are both over all pairs of opposed Bruhat intervals $([v,w],[v',w'])$ of $W$. The unique cell of top dimension $2\ell(w_0)$ is $\framed_{>0}$, indexed by $(W,W)$.
\end{theorem}

For example, the cell decomposition of $\framed_{\ge 0}$ when $G = \SL_2(\CC)$ is shown in \cref{fig_framed_cell_decomposition}, with cells labeled by the corresponding pair of opposed Bruhat intervals $([v,w],[v',w'])$.

\begin{figure}
    \centering
    \begin{tikzpicture}[baseline={([yshift=-.5ex]current bounding box.center)}]
    \begin{scope}[scale=3]
    \draw[blue,fill=green!20] (0,0) -- (1,0) -- (1,1) -- (0,1) -- cycle;
        \draw[dotted,fill=white] (0,1) circle (1pt) (1,0) circle (1pt);
        \fill (0,0) circle (1pt) (1,1) circle (1pt);
    \end{scope}

    \node[green!20!black] at (1.5,1.5) {\footnotesize$([e,s_1],[e,s_1])$};
        
    \node[below left] at (0,0) {$([e,e],[e,e])$};
    \node[above right] at (3,3) {$([s_1,s_1],[s_1,s_1])$};

    \node[blue!50!black, left] at (0,1.5) {$([e,e],[e,s_1])$};
    \node[blue!50!black, right] at (3,1.5) {$([s_1,s_1],[e,s_1])$};
    
    \node[blue!50!black, below] at (1.5,0) {$([e,s_1],[e,e])$};
    \node[blue!50!black, above] at (1.5,3) {$([e,s_1],[s_1,s_1])$};
        
    \end{tikzpicture}
    \caption{The cell decomposition of $\widehat\cT_{\geq 0}$ from \cref{framed_cell_decomposition} for $G=\SL_2(\CC)$. Each cell is labeled by the corresponding pair of opposed Bruhat intervals of $W=\symgp{2}$.}
    \label{fig_framed_cell_decomposition}
\end{figure}
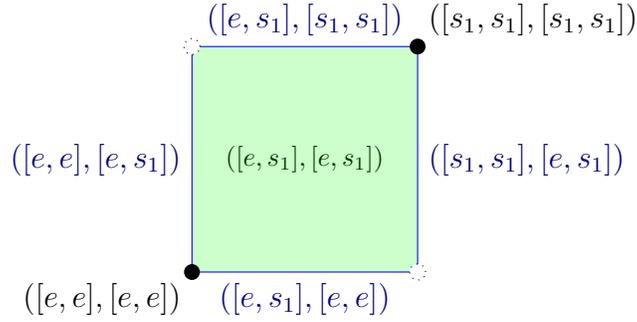

\begin{proof}
The first equality follows from \cref{framed_explicit,opposition_cells}. The homeomorphism in the second line follows from the facts that
\[
\Rtp{v}{w} \times (\Rtp{v'}{w'})^\perp \to \framed_{\ge 0}, \quad (B,B') \mapsto ((B\cap B'), B)
\]
is an injective map and that $\Rtp{v}{w} \cong \RR_{>0}^{\ell(w) - \ell(v)}$. It then follows that the cell indexed by $(W,W)$ is $\framed_{>0}$.
\end{proof}

\begin{corollary}\label{framed_contractible}
The space $\framed_{\ge 0}$ of totally nonnegative framed maximal tori is contractible.
\end{corollary}

\begin{proof}
By \cref{framed_cell_decomposition}, the space $\framed_{\ge 0}$ is homeomorphic to a space sitting between $\cB_{\ge 0} \times \cB_{\ge 0}$ and its interior $\cB_{>0} \times \cB_{>0}$. By \cite{galashin_karp_lam19}, the space $\cB_{\ge 0}$ (and hence also $\cB_{\ge 0} \times \cB_{\ge 0}$) is homeomorphic to a closed ball. Therefore $\framed_{\ge 0}$ is contractible.
\end{proof}

\section{Counterexamples}\label{sec:counterexamples}

\noindent In this section we show explicitly that three results which hold for totally positive spaces do not extend to their totally nonnegative counterparts. We first state the three negative results, and then prove them in the rest of the section using a common example when $G = \SL_3(\CC)$.

First, given $g\in G_{>0}$, Lusztig \cite[Theorem 8.9(a)]{lusztig94} showed that there exists a unique $B\in\cB_{\ge 0}$ containing $g$, and that moreover $B\in\cB_{>0}$. He also showed that conversely, every $B\in\cB_{>0}$ contains an element of $G_{>0}$ \cite[Section 5.6]{lusztig21}. We prove that the latter statement does not extend to $\cB_{\ge 0}$:
\begin{proposition}\label{lusztig_counterexample}
When $G = \SL_3(\CC)$, there exists $B\in\cB_{\ge 0}$ which does not contain any regular semisimple element of $G_{\ge 0}$.
\end{proposition}

In particular, \cref{lusztig_counterexample} provides a counterexample to the conjecture of Lusztig from \cite[Section 5.6]{lusztig21}. Note that the condition on being regular semisimple is necessary to exclude taking $\dot{e}\in G_{\ge 0}$, for example. It is also natural from the perspective of \cite[Section 5]{lusztig21}.

Second, recall from \cref{conjecture_consequence} that every $T\in\cT_{>0}$ contains an element of $G_{>0}$. We show that this property does not extend to $\cT_{\ge 0}$:
\begin{proposition}\label{conjecture_consequence_counterexample}
When $G = \SL_3(\CC)$, there exists $T\in\cT_{\ge 0}$ which does not contain any regular semisimple element of $G_{\ge 0}$.
\end{proposition}

Third, recall from \cref{framed_explicit} that the space of totally nonnegative framed maximal tori is explicitly described as
\[
\framed_{\ge 0} = \{(T,B)\in\framed \mid T\in\cT_{\ge 0}, B\in\cB_{\ge 0}, \textnormal{ and } \oppborel(T,B)\in\cB_{\le 0}\}.
\]
We show that the condition $\oppborel(T,B)\in\cB_{\le 0}$ cannot be omitted in the equality above:
\begin{proposition}\label{framed_explicit_necessary}
When $G = \SL_3(\CC)$, there exists $(T,B)\in\framed$ such that $T\in\cT_{\ge 0}$, $B\in\cB_{\ge 0}$, and $\oppborel(T,B)\notin\cB_{\le 0}$.
\end{proposition}

\begin{proof}[Proof of \cref{lusztig_counterexample,conjecture_consequence_counterexample,framed_explicit_necessary}]
First note that \cref{lusztig_counterexample} implies \cref{conjecture_consequence_counterexample}, by \cref{borel_torus_containment}. Therefore it suffices to prove \cref{lusztig_counterexample,framed_explicit_necessary}.

We set $G = \SL_3(\CC)$ and $g = \begin{bmatrix}
0 & -1 & 0 \\
1 & 0 & 0 \\
0 & 1 & 1
\end{bmatrix}\in G$. We define the Borel subgroups
\begin{align*}
B_1 &= g\cdot B_+, \\[2pt]
B_1' &= g\cdot B_- = g\dot{w}_0\cdot B_+ = \begin{bmatrix}
0 & 1 & 0 \\
0 & 0 & 1 \\
1 & -1 & 0
\end{bmatrix}\cdot B_+, \\[2pt]
B_2 &= g\dot{s}_2\cdot B_+ = \begin{bmatrix}
0 & 0 & 1 \\
1 & 0 & 0 \\
0 & 1 & -1
\end{bmatrix}\cdot B_+, \\[2pt]
B_2' &= g\dot{s}_2 \cdot B_- = g\dot{s}_2\dot{w}_0 \cdot B_+ = \begin{bmatrix}
1 & 0 & 0 \\
0 & 0 & 1 \\
-1 & -1 & 0
\end{bmatrix}\cdot B_+,
\end{align*}
as well as the maximal torus
\[
T = B_1 \cap B_1' = g\cdot T_0 = g\dot{s}_2\cdot T_0 = B_2 \cap B_2'.
\]

By calculating left-justified minors (using the descriptions of $\Fl_3^{\ge 0}$ and $\Fl_3^{\le 0}$ from \eqref{tnn_explicit_Fl3} and \eqref{tnp_explicit_Fl3}, respectively), we can verify that
\[
B_1\in\cB_{\ge 0}, \quad B_1'\in\cB_{\le 0}, \quad B_2\in\cB_{\ge 0}, \quad B_2'\notin\cB_{\le 0}.
\]
In particular, the fact that $T = B_1\cap B_1'$ implies that $T\in\cT_{\ge 0}$ (via \cref{tp_torus_parametrization}). Then \cref{framed_explicit_necessary} follows by taking $B = B_2$, since $\oppborel(T,B_2) = B_2'\notin\cB_{\le 0}$.

It remains to prove \cref{lusztig_counterexample}, which we do by taking $B = B_1$. We must show that every $h\in B_1\cap G_{\ge 0}$ is not regular semisimple (i.e.\ $h$ has a repeated eigenvalue). Since $h\in B_1 = g\cdot B_+$, we can write
\[
h = g\begin{bmatrix}
a & b & c \\
0 & d & e \\
0 & 0 & f
\end{bmatrix}g^{-1} =
\begin{bmatrix}
d-e & 0 & -e \\
-b+c & a & c \\
-d+e+f & 0 & e+f
\end{bmatrix}
\]
for some $a,b,c,d,e,f\in\CC$ with $adf = 1$. Since $h\in G_{\ge 0}$, all minors of $h$ are nonnegative. In particular $a \ge 0$, and since $a\neq 0$ we get $a > 0$. Then considering the left-justified minors in rows $\{3\}$ and $\{2,3\}$ gives $-d+e+f = 0$. Similarly, considering the right-justified minors in rows $\{1\}$ and $\{1,2\}$ gives $e = 0$. Therefore
\[
h = \begin{bmatrix}
d & 0 & 0 \\
-b+c & a & c \\
0 & 0 & d
\end{bmatrix},
\]
so $d$ is a repeated eigenvalue of $h$. This completes the proof of \cref{lusztig_counterexample}.
\end{proof}

\section{Connection to amplituhedra}\label{sec:amplituhedra}

\noindent In this section we explain how the spaces $\cT_{>0}$ and $\cT_{\ge 0}$ naturally arise as complete flag analogues of Grassmannian amplituhedra, which have recently attracted a lot of attention in theoretical physics \cite{arkani-hamed_trnka14}. This provides additional motivation for studying these spaces.

We begin by introducing a Grassmannian analogue of convex polytopes, following Lam \cite{lam16}. (We mention that our definition appears different from Lam's original definition, but is equivalent to it. This is explained in \cite[Sections 3.2 and 9]{karp_williams19}.) Fix $k,m,n\ge 0$ such that $k+m \le n$. Recall from \eqref{positroid_decomposition} that $\Gr_{k,n}^{\ge 0}$ has a cell decomposition into positroid cells. By \eqref{tnp_grassmannian}, taking orthogonal complements gives a cell decomposition of $\Gr_{n-k,n}^{\le 0}$. Let $C\subseteq\Gr_{n-k,n}^{\le 0}$ denote the closure of such a cell, and let $W\in\Gr_{k+m,n}(\CC)$. Then the \newword{Grassmann polytope} (or \emph{Grasstope}) of $W$ and $C$ is defined to be
\begin{align}\label{grassmann_polytope}
\{W\cap V \mid V\in C\} \subseteq \Gr_{m,n}(\CC),
\end{align}
and is only well-defined if every intersection $W\cap V$ has dimension $m$, i.e., $W$ is transverse to every element in $C$.

The problem of determining whether a Grassmann polytope is well-defined was stated as \cref{intro_grassmann_polytope_problem} in the introduction. As we explained there, this is the Grassmann analogue of determining whether two Borel subgroups are opposed to each other (\cref{intro_opposed_problem}). We hope that the techniques of this paper will be useful in studying Grassmann polytopes.

In the case that $C = \Gr_{n-k,n}^{\le 0}$ and $W\in\Gr_{k+m,n}^{>0}$, the Grassmann polytope \eqref{grassmann_polytope} is always well-defined. It is called an \newword{amplituhedron}, denoted $\mathcal{A}_{n,k,m}(W)$. Arkani-Hamed and Trnka \cite{arkani-hamed_trnka14} introduced amplituhedra in their study of scattering amplitudes in high-energy physics, and this motivated Lam to introduce Grassmann polytopes. When $k=1$ amplituhedra are cyclic polytopes in projective space \cite{sturmfels88}, and when $m=1$ each amplituhedron is homeomorphic to the bounded complex of a cyclic hyperplane arrangement \cite{karp_williams19}. Amplituhedra and their triangulations have been extensively studied in the past decade; see \cite{williams23} for a survey.

We are led to define a complete flag analogue of Grassmann polytopes, as follows. Take a Richardson cell $(\Rtp{v'}{w'})^\perp$ of $\cB_{\le 0}$, and let $\overline{(\Rtp{v'}{w'})^\perp}$ denote its Euclidean closure. Also take $B\in\cB$. Then we define the \newword{flagtope}
\begin{align}\label{flagtope_definition}
\{B\cap B' \mid B'\in\overline{(\Rtp{v'}{w'})^\perp}\} \subseteq \cT,
\end{align}
which is \newword{well-defined} if $B$ is opposed to every element in $\overline{(\Rtp{v'}{w'})^\perp}$. We will show (see \cref{flagtope_bijection}) that if $B$ is totally nonnegative, then the flagtope (if it is well-defined) is just homeomorphic to the closed cell $\overline{(\Rtp{v'}{w'})^\perp}$. In particular, the `flag amplituhedron' (obtained when $B\in\cB_{>0}$ and $\overline{(\Rtp{v'}{w'})^\perp} = \cB_{\le 0}$) is just homeomorphic to $\cB_{\le 0}$.

We need the following lemma:
\begin{lemma}\label{subintervals}
Let $[v,w]$ and $[v',w']$ be Bruhat intervals of $W$. Then $[v,w]$ is opposed to every subinterval of $[v',w']$ if and only if $[v',w'] \subseteq [v,w]$.
\end{lemma}

\begin{proof}
($\Rightarrow$) Suppose that $[v,w]$ is opposed to every subinterval of $[v',w']$. Then $[v,w]$ is opposed to $[x,x]$ for all $x\in [v',w']$, so $x\in [v,w]$ by \cref{singleton}. Therefore $[v',w'] \subseteq [v,w]$.

($\Leftarrow$) Suppose that $[v',w'] \subseteq [v,w]$. Then each subinterval of $[v',w']$ intersects $[v,w]$, and hence is opposed to $[v,w]$ by \cref{intervals_intersect}.
\end{proof}

\begin{theorem}\label{flagtope_bijection}
Let $B\in\cB_{\ge 0}$, and let $\overline{(\Rtp{v'}{w'})^\perp}$ be a closed cell of $\cB_{\le 0}$.
\begin{enumerate}[label=(\roman*), leftmargin=*, itemsep=2pt]
\item\label{flagtope_bijection_defined} Let $\Rtp{v}{w}$ denote the Richardson cell containing $B$. Then the flagtope \eqref{flagtope_definition} is well-defined if and only if $[v',w'] \subseteq [v,w]$.
\item\label{flagtope_bijection_homeomorphism} Suppose that the flagtope \eqref{flagtope_definition} is well-defined. Then the map $B' \mapsto B\cap B'$ is a homeomorphism from $\overline{(\Rtp{v'}{w'})^\perp}$ to the flagtope.
\end{enumerate}
\end{theorem}

\begin{proof}
\ref{flagtope_bijection_defined} Recall from \eqref{richardson_closure_tnn} that the cells in the closure of $\Rtp{v'}{w'}$ are indexed precisely by the subintervals of $[v',w']$. Therefore the flagtope is well-defined if and only if $[v,w]$ is opposed to every subinterval of $[v',w']$, which is equivalent to $[v',w'] \subseteq [v,w]$ by \cref{subintervals}.

\ref{flagtope_bijection_homeomorphism} Let $f$ denote the map $B'\mapsto B\cap B'$ from $\overline{(\Rtp{v'}{w'})^\perp}$ to the flagtope. Taking $S = \{B\}$ and $S' = \overline{(\Rtp{v'}{w'})^\perp}$ in \cref{opposed_torus_parametrization} implies that $f$ is bijective. We see that $f$ is continuous, and its inverse is $\oppborel(\cdot,B)$, which is also continuous. Therefore $f$ is a homeomorphism.
\end{proof}

We leave it as an open problem to study flagtopes more generally for partial flag varieties, using the definition of opposition from \cref{sec:opposition_parabolic}:
\begin{problem}
Study the generalization of flagtopes \eqref{flagtope_definition} by replacing $\cB_{\ge 0}$ and $\cB_{\le 0}$ by $\cP_J^{\ge 0}$ and $\cP_{J'}^{\le 0}$, respectively, for arbitrary $J,J'\subseteq I$.
\end{problem}

\bibliographystyle{alpha}
\bibliography{ref}

\end{document}